\documentclass[reqno]{amsart}

\usepackage{tikz}
\usetikzlibrary{decorations.pathreplacing,shapes,arrows}

\usetikzlibrary{decorations.markings}
\usetikzlibrary{arrows,matrix}
\usepgflibrary{arrows}
\tikzset{
    >=stealth',
    punkt/.style={
           rectangle,
           rounded corners,
           draw=black, very thick,
           text width=6.5em,
           minimum height=2em,
           text centered},
    pil/.style={
           ->,
           thick,
           shorten <=2pt,
           shorten >=2pt,}
}
\tikzset{every loop/.style={min distance=2mm,in=225,out=135,looseness=10}}

\usepackage{amssymb,amsfonts,amsmath,amsthm}
\usepackage{mathrsfs}
\usepackage{latexsym, MnSymbol}
\usepackage{graphicx}
\usepackage{color}
\usepackage{verbatim}
\usepackage{enumerate}

\newtheorem{thm}{Theorem}[section]
\newtheorem{pro}[thm]{Proposition}
\newtheorem{lem}[thm]{Lemma}
\newtheorem{cor}[thm]{Corollary}

\theoremstyle{definition}
\newtheorem{dfn}[thm]{Definition}
\newtheorem*{acknowledgements}{Acknowledgements}

\theoremstyle{remark}
\newtheorem{rmk}[thm]{Remark}
\newtheorem*{cla}{Claim}

\newtheorem{que}{Question}

\numberwithin{equation}{section}

\def\J{\mathscr{J}}
\def\D{\mathscr{D}}
\def\R{\mathscr{R}}
\def\L{\mathscr{L}}
\def\H{\mathscr{H}}
\def\T{\mathcal{T}}
\def\E{\mathcal{E}}
\def\B{\mathcal{B}}

\def\es{\varnothing}

\def\ol#1{\overline{#1}}

\def\ig#1{\mathsf{IG}(#1)}
\def\rig#1{\mathsf{RIG}(#1)}
\def\pre#1#2{\langle #1 \; | \; #2 \rangle}
\newcommand{\rrel}{\boldsymbol{r}}

\DeclareMathOperator\im{im}

\begin{document}


\title[The word problem for free idempotent generated semigroups]%
{On regularity and the word problem for \\ free idempotent generated semigroups} 


\author{IGOR DOLINKA}

\address{Department of Mathematics and Informatics, University of Novi Sad, Trg Dositeja Obradovi\'ca 4, 21101 Novi
Sad, Serbia}
\email{dockie@dmi.uns.ac.rs}

\author{ROBERT D. GRAY}

\address{School of Mathematics, University of East Anglia, Norwich NR4 7TJ, England, UK}
\email{Robert.D.Gray@uea.ac.uk}

\author{NIK Ru\v{S}KUC}

\address{School of Mathematics and Statistics, University of St Andrews, St Andrews KY16 9SS, Scotland, UK}
\email{nik.ruskuc@st-andrews.ac.uk}

\thanks{The research of the first author was supported by the Ministry of Education, Science, and Technological Development
of the Republic of Serbia through the grant No.174019, and by the grant No.0851/2015 of 
the Secretariat of Science and Technological Development of the Autonomous Province of Vojvodina.
The research of the second author was partially supported by the EPSRC-funded project EP/N033353/1 `Special inverse monoids: 
subgroups, structure, geometry, rewriting systems and the word problem'.
The research of the third author was supported by the EPSRC-funded project EP/H011978/1 `Automata, Languages,
Decidability in Algebra'.}


\subjclass[2010]{Primary 20M05; Secondary 20F05, 20F10}




\begin{abstract}
The category of all idempotent generated semigroups with a prescribed structure $\E$ of their idempotents $E$ (called the
biordered set) has an initial object called the free idempotent generated semigroup over $\E$,
defined by a presentation over alphabet $E$, and denoted by $\ig{\E}$.
Recently, much effort has been put into investigating the structure of semigroups of the form $\ig{\E}$, especially
regarding their maximal subgroups. In this paper we take these investigations in a new direction by considering the word problem for $\ig{\E}$. 
We prove two principal results, one positive and one negative. 
We show that, for a finite biordered set $\E$, it is decidable whether a given  word $w \in E^*$ represents a regular element; 
if in addition one assumes that all maximal subgroups of $\ig{\E}$
have decidable word problems, then the word problem in $\ig{\E}$
restricted to regular words is decidable. 
On the other hand, we exhibit a biorder $\E$ arising from a finite idempotent semigroup $S$, such that the word
problem for $\ig{\E}$ is undecidable, even though all the maximal subgroups  have decidable word
problems.
This is achieved by relating the word problem of $\ig{\E}$ to the subgroup membership problem in finitely presented groups.
\end{abstract}


\maketitle

\vspace{-4.75mm}
\section{Introduction}
\label{sec:intro}
 
%
%
%

In his foundational paper \cite{Na} Nambooripad made the fundamental observation that the set of idempotents $E(S)$ of an arbitrary semigroup $S$ carries the abstract structure of a so-called biordered set (or regular biordered set in the case of regular semigroups). He provided an axiomatic characterisation of regular biordered sets in his paper. This was later extended by Easdown to arbitrary (non-regular) semigroups \cite{Ea2} who showed that each abstract biordered set is in fact the biordered set of idempotents of a suitable semigroup. Not only are biordered sets important for the study of abstract semigroups but, in addition, for many naturally occurring semigroups their biordered sets of idempotents carry deep algebraic and geometric information. For example, Putcha's theory of monoids of Lie type \cite{Putcha} shows that one can view the biordered set of idempotents of a reductive algebraic monoid as a generalised building, in the sense of Tits. 

The study of biordered sets of idempotents of semigroups is closely related to the study of idempotent generated semigroups. Idempotent-generated semigroups are of interest for a variety of reasons. Firstly, they have the universal property that 
every semigroup embeds into an idempotent generated semigroup \cite{Ho}, and if the semigroup is (finite) countable it can be embedded in a (finite) semigroup generated by $3$ idempotents \cite{Byleen1984}. Secondly, many naturally occurring semigroups have the property that they are idempotent generated. Examples of idempotent generated semigroups include semigroups of transformations \cite{Ho}, matrix semigroups \cite{Er, La}, endomorphism monoids of independence algebras \cite{Gould, FL}, and certain reductive linear algebraic monoids \cite{putcha88, putcha06}. 

Let us suppose that $\E$ is the biordered set arising from the set of idempotents $E$ of a semigroup $S$. The \emph{free idempotent generated semigroup} $\ig{\E}$ is then the free-est semigroup in which the idempotents possess the same structure (for formal definitions see below). In fact, if $S$ is a \emph{regular} semigroup, two such free structures are defined, namely $\ig{\E}$ and its homomorphic image $\rig{\E}$, the \emph{regular free idempotent generated semigroup}
on $\E$. Clearly an important step towards understanding the class of semigroups with a fixed biordered set of idempotents $\E$ is to study 
these free objects. 
The natural question that arises is to which extent and in which ways the structure and the properties of these free objects are determined by those of $S$ and $\E$.

There has been a recent resurgence of interest in the study of free idempotent generated semigroups, see
\cite{BMM1,BMM2,Do1,Do2,DG,DR,ESV,GY,GR1,GR2,McE,YDG}. Most of these recent articles concentrate on maximal subgroups, following in the footsteps of the pioneering work \cite{BMM1} of Brittenham, Margolis and Meakin, where the first non-free such subgroup is exhibited. 
One exception is the article \cite{ESV} by Easdown, Sapir and Volkov, in which the authors investigate the behaviour of elements \emph{not} belonging to maximal subgroups, and show they cannot have finite order. 
The purpose of this paper is to begin the process of broadening this study to what is arguably \emph{the}
key question, namely their word problem and, ultimately, the structure. Until now the word problem for free idempotent generated semigroups has remained poorly understood, with just a handful of  known results which deal only with certain very special classes of biordered sets; see for instance \cite[Section~6]{Pa2}. 
(Note added in revision: Gould and Yang in their recent work \cite{YG} on free idempotent generated semigroups over bands briefly touch on the word problem and prove it is solvable when the underlying band satisfies certain additional conditions.)

Now, for an arbitrary regular biorder $\E$, the semigroup $\rig{\E}$ is structurally very similar to $S$, in terms of their Green's relations, the only difference being in the maximal subgroups. This difference, on the other hand, can be huge, as recent work shows; for example, Dolinka and Ru\v{s}kuc \cite{DR} show that \emph{any} finitely presented group arises as a maximal subgroup of $\rig{\E}$ from a finite semigroup of idempotents (a band). But, due to the tight structural links otherwise, this is the only thing
`that can go wrong': the word problem for $\rig{\E}$, with $E$ finite, is soluble if and only if all the maximal subgroups have soluble word problems (this follows directly from \cite[Theorem 4.1]{Ru}).

The structure of $\ig{\E}$ is much more complicated than that of $\rig{\E}$, as already observed in \cite[Section 3]{BMM1}, where an initial comparison between the two is carried out.
Its regular elements do form a part that in a way `looks like' the regular semigroup $\rig{\E}$, in the sense that the natural epimorphism $\ig{\E}\rightarrow\rig{\E}$ is bijective and Green's structure preserving on this part, and the corresponding maximal subgroups  are isomorphic; see \cite[Theorem 3.6]{BMM1}.
(More formally: the restriction of the above epimorphism to a regular $\D$-class preserves $\R$-, $\L$ and $\H$-classes, and the restriction to a group $\H$-class is a group isomorphism.)
However, $\ig{\E}$ will typically contain non-regular elements as well, and the structure of this part of the semigroup
is not well understood at present.

The aim of this paper is to shift the focus of attention from maximal subgroups to the word problem and the structure of the non-regular part of $\ig{\E}$, by means of establishing the following main results:
\begin{enumerate}
\item
a characterisation for when a word $w\in E^\ast$ represents a regular element in $\ig{\E}$ (Theorem~\ref{thm:testing_reg}); 
in the case where $E$ is finite this characterisation turns into an effective decision procedure (Theorem~\ref{thm:testing_reg_c});
\item
a characterisation for when two words $u,v\in E^\ast$ representing regular elements actually represent the same element of $\ig{\E}$;
when $E$ is finite and all maximal subgroups have soluble word problems, this turns into a solution for the word problem for the regular part of $\ig{\E}$ (Theorem~\ref{thm:reg});
\item
an explicit construction of a finite band $E$ such that all maximal subgroups of $\ig{\E}$ have soluble word problems
but the word problem for $\ig{\E}$ itself is not soluble (Theorem~\ref{thm_BGHMain}).
\end{enumerate}

Our explicit construction is designed so as to relate the word problem in $\ig{\E}$ with the subgroup membership problem in an arbitrary finitely presented group.

As part of (2) we give an algorithm which takes an arbitrary finite biordered set $\E$ and computes Rees matrix representations for each of the regular principal factors of $\ig{\E}$; see Theorem~\ref{thm_presentation} and Lemma~\ref{rho-computable}. These results are important tools which are needed to 
analyse the explicit construction in (3) and establish our results on the word problem. 


Perhaps also worth noting is the immediate corollary from the discussion above and our main results that there exists a finite regular biordered set $\E$ such that the word problem for $\rig{\E}$ is decidable, while that for $\ig{\E}$ is not.

The paper is organised as follows.
In Section~\ref{sec_HowieLallement} we present some results about the action of idempotents on the $\H$-classes in a fixed $\D$-class of a semigroup, and show how this action is encoded in the biorder of an idempotent generated semigroup. These ideas are then applied in Section~\ref{sec_WPReg} where our results on the word problem for regular elements of $\ig{\E}$ are given. Section~\ref{sec_RM} contains results which show how to compute Rees matrix representations of regular $\D$-classes of $\ig{\E}$. In Section~\ref{sec_undec1} we outline the main ideas behind, and state our main results for, our explicit construction of a finite band $E$ such that all maximal subgroups of $\ig{\E}$ have soluble word problems but the word problem for $\ig{\E}$ itself is not soluble. The details of our construction are given in Section~\ref{sec_BGH}, and the proofs of our main results about this construction are then obtained in Sections~\ref{sec_maxsub} and \ref{sec:undec}. Finally in Section~\ref{sec_conc} we make some concluding remarks, and discuss possible  future research directions that arise  from this work.

\section{Actions of idempotents on $\H$-classes}\label{sec_HowieLallement}

In this section we prove some fundamental results about the way that idempotents act on the $\H$-classes within a fixed $\D$-class of a semigroup. Our interest ultimately is in showing that for idempotent generated semigroups these actions are encoded by the biordered set of idempotents of the semigroup. Before turning our attention to idempotent generated semigroups we begin with some general results that hold for arbitrary semigroups. For this it will be useful to first recall a few basic concepts from semigroup theory. Further background in semigroup theory may be found in \cite{CP,Hi,HoBook,LaBook}.

Substantial information about a semigroup may be gained by studying its ideal structure. One of the most fundamental tools in this regard are the five equivalence relations called \emph{Green's relations}. Given a semigroup $S$, we
define for $a,b\in S$:
$$
a\;\R\; b \Leftrightarrow aS^1=bS^1, \quad a\;\L\; b \Leftrightarrow S^1a=S^1b,\quad a\;\J\; b \Leftrightarrow
S^1aS^1=S^1bS^1,
$$
where $S^1$ denotes $S$ with an identity element adjoined (unless $S$ already has one); hence, these three relations
record when two elements of $S$ generate the same principal right, left, and two-sided ideals, respectively.
Furthermore, we let $\H=\R\cap\L$ and $\D=\R\circ\L$, and remark that $\D$ is the join of the equivalences $\R$ and $\L$ because $\R\circ\L=\L\circ\R$. As is
well known, for finite semigroups we always have $\D=\J$, while in general the inclusions
$\H\subseteq\R,\L\subseteq\D\subseteq\J$ hold. The $\R$-class of $a$ is denoted by $R_a$, and in a similar
fashion we use the notation $L_a,J_a,H_a$ and $D_a$.

Recall that an element $a$ of a semigroup $S$ is \emph{(von Neumann) regular} if there exists $a'\in S$ such that
$aa'a=a$. 
If $a$ is an element of a semigroup $S$ we say $a'$ is an \emph{inverse} of $a$ if $aa'a=a$ and $a'aa'=a'$. 
Note that an element with an inverse is necessarily regular. In fact, the converse is also true: every regular element has an inverse.  
%
%
%
%
It is well known that a single $\D$-class consists either entirely of regular or non-regular elements; see \cite[Proposition 2.3.1]{HoBook}. Therefore, regular $\D$-classes are precisely those containing idempotents, and
for each idempotent $e$, the $\H$-class $H_e$ is a group with identity $e$. In fact, this is a maximal subgroup of the
semigroup under consideration and all maximal subgroups arise in this way.

The following result is due to Howie and Lallement \cite[Lemma~1.1]{HL}. 

\begin{lem}[Howie--Lallement Lemma]
Let $S$ be a semigroup and let $e,f \in E(S)$. If $ef \in S$ is regular then $ef$ has an idempotent inverse $g$ such that $ge=g$ and $fg=g$. 
\end{lem}
\begin{proof}
We outline the proof 
here for completeness. 
Since $ef$ is regular it has an inverse $x$ in $S$ satisfying
$
ef x ef = ef$, and
$x ef x = x$. 
Then $(fxe)^2 = f(xefx)e = fxe$, so $fxe$ is an idempotent. 
Now routine calculations show the result holds by taking $g = fxe \in E(S)$. 
%
%
%
%
%
\end{proof}

Applying the Howie--Lallement Lemma we obtain the following crucial general result which describes the way that idempotents can act on the $\H$-classes in a given fixed $\R$-class of a semigroup. This is illustrated in Figure~\ref{fig_IDActions}. 

%
%
%
%

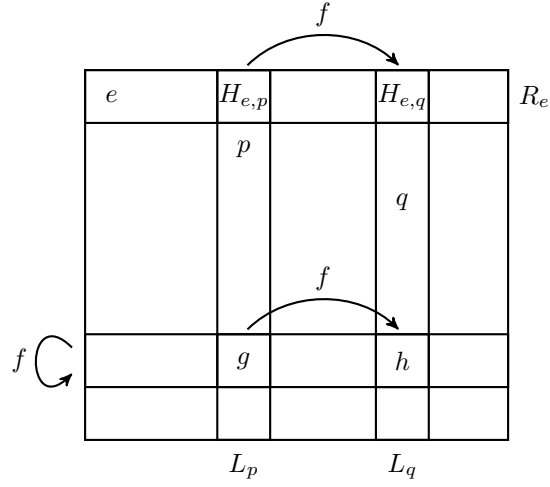
\begin{figure}
\begin{center}
\scalebox{1}
{
\begin{tikzpicture}
[scale=1, 
SHclass/.style ={rectangle, draw=black, thick, fill=blue!20, minimum height=2em, minimum width=2em},
Hclass/.style ={draw, rectangle, thick, minimum height=2em, minimum width=2em},
Drectangle/.style ={draw, rectangle, thick, minimum height=14em, minimum width=16em},
Lrectangle/.style ={draw, rectangle, thick, minimum height=14em, minimum width=2em},
Rrectangle/.style ={draw, rectangle, thick, minimum height=2em, minimum width=16em}],
\node (Hep) at (-2em,6em) [Hclass] {};
\node (Heq) at (4em,6em) [Hclass] {};
\node (Hg) at (-2em,-4em) [Hclass] {$g$};
\node (Hh) at (4em,-4em) [Hclass] {$h$};
%
\node (Hleft2) at (-8em,-4em) {};
\path[->] (Hleft2) edge  [pil, loop left, thick] node {$f$} ();
%
%
%
%
\path[-stealth] (Hep.north) edge[pil, bend left=45] node [above] {$f$} (Heq.north);
\path[-stealth] (Hg.north) edge[pil, bend left=45] node [above] {$f$} (Hh.north);
\node (D) at (0em,0em) [Drectangle] {};
%
%
\node (Re) at (0em,6em) [Rrectangle] {};
\node at (-7em,6em) {$e$};
\node (Lp) at (-2em,0em) [Lrectangle] {};
\node at (-2em,6em) {$H_{e,p}$};
\node at (-2em,-8em) {$L_p$};
\node at (-2em,4em) {$p$};
\node (Lq) at (4em,0em) [Lrectangle] {};
\node at (4em,6em) {$H_{e,q}$};
\node at (4em,-8em) {$L_q$};
\node at (4em,2em) {$q$};
\node at (9em,6em) {$R_{e}$};
\node (Rg) at (0em,-4em) [Rrectangle] {};
%
%
%
%
%
%
%
%
\end{tikzpicture}
}
\end{center}
\caption{
The action of an idempotent on the $\H$-classes within a fixed $\R$-class in an arbitrary semigroup, as described by Proposition~\ref{prop_pure}. 
}\label{fig_IDActions}
\end{figure}

\begin{pro}\label{prop_pure}
Let $S$ be a semigroup, let $e, p, q \in E(S)$ be $\D$-related idempotents in $S$, and let $f \in E(S)$. Set 
\[
H_{e,p} = R_e \cap L_p, \quad \mbox{and} \quad 
H_{e,q} = R_e \cap L_q. 
\]
Then 
\[
H_{e,p} f = H_{e,q} 
\]
if and only if there exist idempotents $g, h \in E(S)$ such that 
\[
p \L g \R h \L q, \quad
fg=g, \quad \mbox{and} \quad
gf=h. 
\]
\end{pro}

\begin{rmk}
The reader should note the complication of similar looking, but fundamentally different, notations for $\H$-classes: the single-indexed 
$H_a$ stands for the $\H$-class of the element $a$, while the double-indexed $H_{e,p}$ is the intersection of the $\R$-class of $e$ with the $\L$-class of $p$.
\end{rmk}

\begin{proof}[Proof of Proposition \ref{prop_pure}]
($\Rightarrow$) Suppose that $H_{e,p} f = H_{e,q}$ in $S$. By Green's Lemma \cite[Lemma~2.2.1]{HoBook} the map $x \mapsto xf$ is an $\R$-class preserving bijection from $L_p$ to $L_q$. It follows that in $S$ we have $p \R pf \L q$. In particular, $pf$ is a regular element of $S$. By the Howie--Lallement Lemma $pf$ has an idempotent inverse $g$ such that $gp=g$ and $fg=g$. Since $pf$ and $g$ are inverses it follows that $pg = (pf)g$ is an idempotent with $pf \R pfg = pg \L g$. Consider the idempotent $pg \in E$. From above 
$g\L pg\R pf\R p$, and since $pg$ is an idempotent it follows by the Miller--Clifford Theorem \cite[Proposition~2.3.7]{HoBook} that $g=gp \in R_g \cap L_p$. Set $h = gf$ which is an idempotent $\R$-related to $g$, because $fg=g$. 
Since $g \in L_p$ and $x \mapsto xf$ is an $\R$-class preserving bijection from $L_p$ to $L_q$, it follows that $h = gf \L pf \L q$. 
Therefore we have found idempotents $g, h \in S$ such that  $p \L g \R h =gf \L q$,
$fg=g$ and $gf = h$. 

($\Leftarrow$) Now suppose that there exist $g, h \in E(S)$ such that $p \L g \R h \L q$,
$fg=g$ and $gf = h$. 
By Green's Lemma the mapping $x\mapsto xf$ is an $\R$-preserving bijection
between $L_g$ and $L_h$
Since $H_{e,p}=R_e\cap L_p=R_e\cap L_ g$ and $H_{e,q}=R_e\cap L_q=R_e\cap L_h$,
it follows that $H_{e, p}f = H_{e, q}$, as required.
%
%
%
%
\end{proof} 




\begin{dfn}\label{def_biorder}
Let $S$ be a semigroup with set $E = E(S)$ of idempotents. 
The \emph{biordered set} $\E = (E, *)$ is the partial algebra where 
\begin{equation}
\label{eqRev5}
e * f =   ef \ 
\mbox{if $\{e,f\} \cap \{ef, fe \} \neq \varnothing$,}
\end{equation} 
and the product $e * f$ in $\E$ is undefined otherwise. 
A pair of idempotents $(e,f)$ satisfying the condition $\{e,f\}\cap \{ef,fe\}\neq \varnothing$
is called a \emph{basic pair}.
\end{dfn}
It is easy to show that for such a pair their product $ef$ is indeed again an idempotent, and that it is either $\R$-related to $e$ or $\L$-related to $f$. Throughout we shall abuse notation writing $ef$ instead of $e * f$. The reader should remember that if the product $ef$ is indicated to take place in $\E = \E(S)$ then Condition \eqref{eqRev5} must hold.

We now turn our attention to idempotent generated semigroups. In what follows $S$ will denote an 
idempotent generated semigroup with set of idempotents $E = E(S)$ and corresponding biorder $\E = \E(S)$.


\begin{lem}\label{lem_zero} For all $e, f \in E$ we have 
\begin{align}
\label{eqRev1}
\mbox{$e \R f$ in $S$} & \Leftrightarrow  \mbox{$ef=f$ and $fe=e$ in $\E$}, \\
\label{eqRev2}
\mbox{$e \L f$ in $S$} & \Leftrightarrow  \mbox{$ef=e$ and $fe=f$ in $\E$}, 
\end{align}
and
\begin{equation}
\label{eqRev3}
\mbox{$e \D f$ in $S$}  \Leftrightarrow  
\exists \ e_1, e_2, \ldots, e_n \in E: e = e_1 \R e_2 \L e_3 \R \ldots \L e_{n-1} \R e_n =f \ \mbox{in $S$}. 
\end{equation}
\end{lem}
\begin{proof}
This first two claims are immediate from the definition of $\E$. The third clause can be recovered by appealing to the theory of $E$-chains \cite{Na}, but for the sake of completeness we show how it can be proved by applying Proposition~\ref{prop_pure}. 
Suppose $e\D f$, and let
$t \in S$ be such that $e \R t \L f$. As $S$ is idempotent generated we can write $t$ as a product of idempotents 
$t = e_1 e_2 \ldots e_k$. 
From $e\R f$ it follows that $et=t$, and hence
\[
e \R e e_1 \R e e_1 e_2 \R \ldots \R e e_1 e_2 \ldots e_k. 
\]
Therefore for each $1 \leq i \leq k$, 
\[
H_{e e_1 e_2 \ldots e_{i-1}} e_i =
H_{e e_1 e_2 \ldots e_{i}},
\]
and, applying Proposition~\ref{prop_pure}, there are idempotents $g_i, h_i$ with 
\[
e e_1 e_2 \ldots e_{i-1} \L g_i \R h_i \L e e_1 e_2 \ldots e_{i}. 
\]
But then 
\[
e \L g_1 \R h_1 \L g_2 \R h_2 \ldots \L g_k \R h_k \L f, 
\]
as required.  The converse implication is obvious.
%
%
\end{proof}

The following lemma is another application of Proposition~\ref{prop_pure} and will show how in $S$ the action of the elements from $E$ on the $\H$-classes in a given regular $\R$-class is completely determined by (and is computable from) the biorder $\E$. Even though not explicitly 
referred to 
 here, the ideas in the following lemmas intimately relate to the notion of sandwich set (and generalised sandwich sets) and their connection to the theory of biordered sets, as explored by Pastijn in \cite{Pa2}. 


\begin{lem}\label{lem_ExtremeHowie}
Let $e, p, q \in E$ be $\D$-related idempotents in $S$, and let $f \in E$. Set 
\[
H_{e,p} = R_e \cap L_p, \quad \mbox{and} \quad 
H_{e,q} = R_e \cap L_q. 
\]
Then 
\[
H_{e,p} f = H_{e,q} \ \mbox{in $S$}
\]
if and only if there exist $g, h \in E$ such that the following equations all hold in the biorder $\E$:
\begin{align*}
pg &= p, 			&	gp &= g,  &
gh &= h, 			&	hg &= g, \\ 
hq &= h, 			&	qh &= q, &
fg &= g, 				&	gf &= h.  
\end{align*}
\end{lem}
\begin{proof}
($\Rightarrow$) Suppose that $H_{e,p} f = H_{e,q}$ in $S$. By 
Proposition~\ref{prop_pure} 
there exist idempotents 
$g, h \in E$ such that  in $S$ we have $p \L g \R h \L q$,
$fg=g$ and $gf = h$. In particular $(f,g)$ is a basic pair of idempotents,
yielding $fg=g$, $gf=h$ in $\E$.
The remaining six equalities express $p\L g\R h\L q$, as in Lemma \ref{lem_zero}.  

($\Leftarrow$) Now suppose that there exist $g, h \in E$ such that the eight listed equations hold in $\E$. It follows that the equations also hold in $S$, and thus in $S$ we have $p \L g \R h \L q$, $fg=g$ and $gf=h$. 
Now the result follows by the converse implication of Proposition~\ref{prop_pure}.
\end{proof}

\begin{rmk}
Lemma~\ref{lem_ExtremeHowie} is general, and in particular no finiteness assumption is imposed on the set $E$. 
It is important to stress that the action we are considering of $S$ on the
$\H$-classes in a given fixed $\R$-class is different from the closely related 
action of $S$ on the $\L$-classes in $\D$ (even though the two coincide in the finite case).
To see this consider the bicyclic monoid $B=\langle b,c\:|\: bc=1\rangle$.
In the action on the $\H$-classes of $R_1$, the product $H_{1,1}c$ is undefined, while in the action on $\L$-classes we have $L_1c=L_1$.
\end{rmk} 

The following result formalises the statement that the action of $E$ on the $\H$-classes in a given $\R$-class of $S$ is completely determined by the biorder $\E(S)$. 

\begin{lem}\label{lem_MatchingAction}
Let $S$ and $T$ be semigroups with the same biordered set of idempotents $\E(S)=\E(T)=\E$. Let $e\in E$ and $w \in E^*$. Then
\[
\mbox{
$ew \R e$ in $S$ $\Leftrightarrow$
$ew \R e$ in $T$. 
}
\]
Additionally, if both are true, then for any $q \in E$ we have 
\[
\mbox{
$ew \L q$ in $S$ $\Leftrightarrow$ 
$ew \L q$ in $T$. 
}
\]
\end{lem}
\begin{proof}
Both claims are proved simultaneously by induction on the length of the word $w$, using Lemma~\ref{lem_ExtremeHowie}, which implies that the validity of each of the four constituent  clauses
(namely, $ew \R e$ in $S$, $ew \R e$ in $T$, $ew \L q$ in $S$ and $ew \L q$ in $T$)
depends only on $\E$, and not on $S$ or $T$ themselves.
%
\end{proof}

\section{The word problem for regular elements of $\ig{\E}$}\label{sec_WPReg}
\label{secWPreg}

Given an arbitrary biordered set $\E = \E(S)$, we will define the \emph{free idempotent generated semigroup} $\ig{\E}$ associated to $\E$ 
by means of a presentation.
The set of generating symbols in this presentation will be precisely $E$, which opens up the danger of ambiguity in what follows: products of elements of $E$ can be interpreted both as words over $E$ (i.e. elements of $E^\ast$) or as the specific elements of $E$ (or indeed of $S$) to which they evaluate.
With this in mind,  $\ig{\E}$
is the semigroup defined by the presentation 
\begin{equation} \label{eq:igE}
\ig{\E} = \pre{E}{\,ef=e * f \ (\{e,f\} \cap \{ef,fe\} \neq \es)\,},
\end{equation}
where $ef$ is the word of length $2$ with letters $e$ and $f$, while $e * f$ denotes the partial multiplication in $\E$ considered as a partial algebra; see Definition~\ref{def_biorder}. 
When working with this presentation, given two words $u, w \in E^*$ we shall write $u \equiv v$ to mean $u$ and $v$ are identical as words in $E^*$, and write $u=v$ to mean they represent the same element of the semigroup $\ig{\E}$. 

Clearly, $\ig{\E}$ is an idempotent generated semigroup, and
it follows from \cite{Ea2} that  $\E(\ig{\E}) = \E$, that is, the biordered set of idempotents of $\ig{\E}$ is precisely $\E$. 
In particular, Lemmas~\ref{lem_zero},
\ref{lem_ExtremeHowie}, and \ref{lem_MatchingAction} all apply to this semigroup. 
%
%
In the special case that $\E$ is finite, we can deduce the following decidability results for the semigroup $\ig{\E}$. 

\begin{lem}\label{lem_zero_c}
There is an algorithm which takes a finite biordered set $\E$ and $e, f \in E$ and decides each of $e \R f$, $e \L f$, and $e \D f$  in $\ig{\E}$.  
\end{lem}
\begin{proof}
This is an immediate consequence of Lemma~\ref{lem_zero}:
for the relations $\R$ and $\L$ the criteria in \eqref{eqRev1}, \eqref{eqRev2} can be read off directly from $\E$, and for 
$\D$ one has to note that the sequence $e_1,\dots, e_n$ in \eqref{eqRev3} can be taken not to contain any repeats, and then its length is bounded by $|\E|$.
\end{proof}

\begin{rmk}
Let us point out the subtle difference between a finite biorder and a biorder arising from a finite semigroup.
Clearly, the biorder arising from a finite semigroup is finite. However the converse is not true, as demonstrated by Easdown; see \cite[Theorem 13]{Ea1}.
All our decidability results are predicated on finiteness, and we couch them in the more general setting of finite biordered sets. 
%
\end{rmk}

\begin{lem}\label{lem_three_c}\label{lem_four_c}
There exist algorithms which for any given finite biordered set $\E$ decide the following:
\begin{itemize}
\item[\textup{(i)}]
for given $\D$-related idempotents $e, p, q \in E$, 
and another idempotent $f \in E$, 
whether  
\[
H_{e,p} f = H_{e,q} \ \mbox{in $\ig{\E}$},
\]
where 
\[
H_{e,p} = R_e \cap L_p \quad \mbox{and} \quad 
H_{e,q} = R_e \cap L_q;
\]
\item[\textup{(ii)}]
for elements $e \in E$ and $w \in E^*$ whether 
$ew \R e$ in $\ig{\E}$, and if so returns $q \in E$ such that $ew \L q$. 
\end{itemize}
\end{lem}

\begin{proof}
\textup{(i)}
This is an immediate consequence of Lemma~\ref{lem_ExtremeHowie}. 

\textup{(ii)}
If $w \equiv \epsilon$ the algorithm gives an affirmative answer, and returns $q = e$. 
Otherwise write $w \equiv w'f$. From the definition of Green's relation $\R$, and the fact that in a regular $\D$-class every $\L$-class contains an idempotent, it follows that in $\ig{\E}$ we have
\[
ew \R e \Leftrightarrow 
ew' \R e \ \mbox{and} \ 
H_{e,q'}f = H_{e,q}, 
\]
for some idempotents $q', q \in E$ with $q' \L ew'$ and $q \L ew$. Recursively check whether $ew' \R e$ and, if so, compute all
$q' \in E$ with $q' \L ew'$. Then use Lemma~\ref{lem_three_c} to check if there exists $q \in E$ such that $H_{e,q'}f = H_{e,q}$,
and if so return an affirmative answer and $q$. 
\end{proof}
We now turn our attention to the word problem for regular elements of $\ig{\E}$. 
First we prove a lemma which describes the form that words representing regular elements of $\ig{\E}$ can take.  
\begin{lem}\label{lem_RegWords}
Suppose that $u, v \in E^*$, $e \in E$ satisfy $ue \L e$, $ev \R e$ in $\ig{\E}$. 
Then applying a single relation from the presentation for $\ig{\E}$ to the word $uev$ yields a word $u'e'v'$ such that
\[
u'e' \L e' \R e'v'
\]
holds in $\ig{\E}$.
In particular, $e' \D e$ in $\ig{\E}$. 
\end{lem}
\begin{proof}
If a relation is applied to $u$, yielding a word $u'$, set $e'=e$, $v'\equiv v$, and the assertion is obvious. The case of applying a relation to $v$ is analogous. 

Suppose now we apply a relation of the form $e=fg$ to $e$. Then we know that either 
$e \R f$ or $e \L g$. Without loss of generality, suppose $e \R f$. 
In this case set $u'\equiv u$, $e' = f$ and $v' \equiv gv$. Firstly, we have 
\[
e'v' \equiv fgv = ev \R e \R f = e',
\]
Secondly, from $ue \L e$ and Green's Lemma \cite[Lemma 2.2.1]{HoBook} we have that $x \mapsto ux$ is an $\L$-class preserving bijection $R_e \rightarrow R_{ue}$. But $f \in R_e$, and so 
\[
u'e' \equiv uf \L f = e',
\]
as required. 

For the next case, suppose that $v \equiv fv_1$, that $(e,f)$ is a basic pair, and the relation applied is $ef = g$. 
From $e \R ev \equiv efv_1$ we have $e \R ef = g$, and the result may be seen to hold by setting $u'\equiv u$, $e' = g$ and $v' \equiv v_1$. Finally, the case where $u \equiv u_1 f$ and the relation applied is $fe=g$ is dual to this one. 

This covers all possible applications of relations, and the proof is complete. 
\end{proof}
An immediate corollary of this lemma is the following result. 
\begin{lem}\label{lem_RegWords2}
Suppose that $u, v \in E^*$, $e \in E$ satisfy $ue \L e$, $ev \R e$ in $\ig{\E}$. 
Then every word over the alphabet $E$ which is equal in $\ig{\E}$ to the word $uev$ has the form $u'e'v'$ where 
\[
u'e' \L e' \R e'v',
\]
and $e \D e'$. 
\end{lem}

\begin{thm}\label{thm:testing_reg}
Let $\E$ be an arbitrary biordered set. 
A word $w \in E^*$ represents a regular element of $\ig{\E}$ if and only if $w \equiv uev$ where 
$e \in E$ and 
$ue \L e \R ev$ in $\ig{\E}$, in which case $e \D w$ in $\ig{\E}$. 
%
%
\end{thm}
\begin{proof}
($\Leftarrow$) Since $\R$ is a left congruence, we deduce
$
uev \R ue \L e, 
$
and hence $uev$ represents a regular element. 

($\Rightarrow$) Suppose that $w \in E^*$ represents a regular element of $\ig{\E}$. 
Then there is an idempotent $e \in E$ with $e \R w$. Then $w = ew$ in $\ig{\E}$. Now $ew$ has the form 
given in the statement of Lemma~\ref{lem_RegWords2}, with $u\equiv \epsilon$ and $v\equiv w$. Therefore, by Lemma~\ref{lem_RegWords2}, the word $w$ also has this form. 

The last clause follows easily since, by the 
Miller--Clifford Theorem, 
$ue \L e \R ev$ implies $e \D (ue)(ev) = uev$ in $\ig{\E}$. 
\end{proof}

In the special case of finite biordered sets $\E$, combining this theorem with the decidability results already obtained above yields the following result, which is the first main result of this section. 

\begin{thm}\label{thm:testing_reg_c}
There is an algorithm which takes a finite biordered set $\E$ and a word $w \in E^*$ and decides whether $w$ represents a regular element of $\ig{\E}$, and if so, returns $e, f \in E$ such that $e \R w \L f$ in $\ig{\E}$. 
\end{thm}

\begin{proof}
For every decomposition $w=uev$ ($e\in E$, $u,v\in E^\ast$) check whether  $ev\R e$ and $ue\L e$ using 
Lemma \ref{lem_zero_c}.
By Theorem \ref{thm:testing_reg}, the element represented by $w$ is regular if and only if the  answer is affirmative for some decomposition. 
In this case, the requisite idempotents $e$ and $f$ can be computed by
repeatedly, letter by letter, running the algorithms from Lemma \ref{lem_three_c} (i) and (ii) on words $ev$ and $ue$ respectively. 
\end{proof}

Theorem~\ref{thm:testing_reg_c} can be interpreted as saying 
that there is an algorithm which takes an arbitrary finite biordered set, tests regularity of words,   
and for any given regular word can identify the `position' of the $\H$-class of that element in $\ig{\E}$. We would like to extend this to a solution to the word problem for regular elements
but this will only be possible under the assumption that the word problem can be solved for the maximal subgroups of $\ig{\E}$, so we now turn our attention to them. We pass from $\ig{\E}$ to its maximal subgroups using a Reidemeister--Schreier type rewriting process as first described in \cite{GR1}. We shall recall the details of this process here. In the special case of finite biordered sets, we shall explain how this process 
turns into 
 an algorithm for writing down presentations for maximal subgroups of $\ig{\E}$,
 %
 %
and as a byproduct shall obtain an algorithm which gives Rees matrix representations for each of the regular $\D$-classes of $\ig{\E}$.

In what follows, $S$ will continue to stand for an idempotent generated semigroup, with the set of idempotents $E=E(S)$ and the corresponding biorder $\E=\E(S)$, and
let us fix an arbitrary $e \in E$. Let $D$ be the $\D$-class of $e$ in $\ig{\E}$. 
Let $R_i \; (i \in I)$ and $L_j \; (j \in J)$ be the sets of $\R$- and $\L$-classes respectively in $D$.
We denote the $\H$-classes in $D$ by $H_{i,j} = R_i \cap L_j$.
Unfortunately, this is yet another indexing of $\H$-classes that the reader will have to contend with -- on the plus side, it is compatible with the earlier $H_{e,p}$ notation, provided the indexing sets $I$ and $J$ are taken to consist of representatives.
Let 
\[
K = \{ (i,j) \::\: \mbox{$H_{i,j}$ is a group} \},
\]
 and denote by $e_{i j}$ the unique idempotent in $H_{i,j}$ ($(i,j)\in K$). 
For notational convenience assume that the symbol $1$ belongs to both index sets $I$ and $J$,
and that $e=e_{11}$.
Consider the $\R$-class $R = R_1$ and its $\H$-classes $H_{1,j} \ (j \in J)$. 
We have seen that the generators $E$ act on this set of $\H$-classes,
and we now translate this action to 
an action $(j,w) \mapsto jw$ of the free monoid $E^*$ on the index set $J \cup \{ 0 \}$.
Specifically, for each letter $g \in E$, we set
\begin{align*}
(j,g) & \mapsto 
\begin{cases}
l  &  \mbox{if $H_{1,j} g = H_{1,l}$ in $\ig{\E}$}, \\
0 & \mbox{otherwise},
\end{cases}
\end{align*}   
and extend by freeness.
By Lemma~\ref{lem_MatchingAction} the $\H$-classes in the $\R$-class of $e$ in $\ig{\E}$ are in natural bijective correspondence with the $\H$-classes in the $\R$-class of $e$ in $S$. If we index the latter as $H_{1,j}^S \ (j \in J)$ in the natural way we have
\[
H_{1,j}^S g = H_{1,jg}^S \ \mbox{if and only if} \ jg \neq 0. 
\]
This reflects the fact that the action of $E$ on the $\H$-classes inside an $\R$-class is entirely determined by the biorder $\E$ and does not depend on the actual semigroup. 

Arguing as in \cite{GR1}, there exist words $r_j \in (E \cap D)^*$ for $j  \in J$, satisfying $1r_j=j$ and every prefix of every $r_j$ is equal to some $r_l$. We call this set of words a \emph{Schreier system}. Furthermore, for each $j\in J$, there exists a word $r_j' \in (E \cap D)^*$ such that
$j r_j' = 1$ and
\begin{align*}
x r_j r_j' &= x \quad \mbox{in $\ig{\E}$ for all $x \in H_{1,1}$, and} \\
x r_j' r_j &= x \quad \mbox{in $\ig{\E}$ for all $x \in H_{1,j}$.} 
\end{align*}
In fact, such words  $r_j$, $r_j'$ $(j \in J)$ can be determined solely from the biorder $\E$ in the following inductive way. To begin with set $r_1 \equiv r_1' \equiv \epsilon$, $J_1 = \{ 1 \}$. Suppose that at step $k$ 	a set $J_k\subseteq J$ is computed, and for every $j \in J_k$ words $r_j$, $r_j'$ satisfying the required properties are also computed. 
Suppose $J_k\neq J$.
For every $j \in J_k\setminus J_{k-1}$ and every $f \in E$ with 
$0 \neq jf \in J \setminus J_k$
by Proposition~\ref{prop_pure} there exist idempotents $g, h \in E$ such that $g \in L_j$, $h \in L_{jf}$, $g \R h$, $fg=g$, and $gf=h$. Now let 
$r_{jf} \equiv r_j h$, $r_{jf}' \equiv gr_j'$, and add $jf$ to $J_{k+1}$. 
That every $j\in J$ will be reached in the course of this procedure follows from the definition of $\R$ and the fact that $E$ is a generating set.
Hence, when $E$ is finite this becomes an actual algorithm to compute the $r_j$, $r_j'$ $(j \in J)$ from the biorder $\E$.

Using these definitions we shall now work towards a presentation for the group $H = H_{1,1}$. 
By results from \cite{Ru}, the elements of $\ig{\E}$ represented by words 
$e r_j f r_{jf}' \in E^*$ with $j \in J$, $f \in E$, $jf \neq 0$,
form a generating set of $H$. Motivated by this, introduce a new alphabet
\[
B= \{ [j,f]: j \in J, f \in E, jf \neq 0 \}.
\]
Furthermore, let $\psi:B^* \rightarrow E^*$ be the unique homomorphism extending
$$
[j,f] \mapsto e r_j f r_{j f}'.
$$
Now define
\[
\phi: \{ (j,w):\ j \in J, w \in E^*, j w \neq 0 \} \rightarrow B^*
\]
by $\phi(j,1) \equiv 1$ and, for $w \equiv e_1 \cdots e_t$ with $t \geq 1$,
\[
\phi(j,e_1\cdots e_t) \equiv [j, e_1] [j e_1, e_2] [j e_1e_2, e_3] \cdots [j e_1 e_2 \cdots e_{t-1}, e_t] \in B^*.
\]
As a straightforward consequence of this definition we have
\begin{equation}\label{eqn:hom2}
\phi(j,w_1w_2) \equiv \phi(j,w_1)\phi(jw_1,w_2) \quad (j \in J, w_1, w_2 \in E^*, j w_1 w_2 \neq 0).
\end{equation}
Note that $\phi$ rewrites a pair $(j,w)$ into a word over $B$ whose image under $\psi$ is a word over $E$ representing the element
$e r_j w r_{jw}'$ in $\ig{\E}$. In particular, if $w \in E^*$ represents an element of $H$ then we have
\begin{equation}
\label{eqRev4}
\psi \phi(1,w) = w
\end{equation}
in $\ig{\E}$. 

\begin{thm} \label{thm:prescomp} 
Let $\mathfrak{R}$ denote the set of relations in the presentation \eqref{eq:igE} of $\ig{\E}$. Then, with the above notation,
a monoid presentation for the maximal subgroup group $H = H_e$ of $\ig{\E}$ is given by
\begin{align*}
\langle \, B \ | \ & \phi(j,\alpha) = \phi(j,\beta)    & &(j \in J, \ (\alpha=\beta) \in \mathfrak{R}, \ j \alpha \neq 0), \\
        & \phi(1, er_j a r_{j a}') = [j,a]   & &(j \in J, \ a \in E, \ j a \neq 0),\\
        & \phi(1,e) = 1 \, \rangle.
\end{align*}
If the biorder $\E$ is finite then so is the above presentation, and there is an algorithm to compute the presentation from $\E$ and any given $e\in E$.
\end{thm}

\begin{proof}
That the presentation defines $H=H_e$ is proved in \cite[Theorem 2.9]{Ru},
and it clearly is finite when $\E$ is. The existence of an algorithm to compute the presentation when $\E$ is a finite biordered set follows from the fact that 
the action of $\ig{\E}$ on $\{H_{1,j}\::\: j\in J\}$ is computable in the sense of Lemma \ref{lem_three_c}.  
%
\end{proof}


\begin{lem}\label{lem:WP}
Let $w_1, w_2 \in E^*$ such that $e \R w_1 \H w_2$ 
%
%
in $\ig{\E}$. Then $w_1 = w_2$ in $\ig{\E}$ if and
only if $\phi(1,w_1) = \phi(1,w_2)$ in $H$.
\end{lem}

\begin{proof}
The direct implication  is an immediate consequence of \cite[Lemma~2.10]{Ru}.
For the converse,
suppose $\phi(1,w_1) = \phi(1,w_2)$ where $w_1, w_2 \in H_{1,j}$ and $e \R w_1$.
Then
\begin{align*}
 \phi(1,w_1) = \phi(1,w_2)
  & \Rightarrow   \phi(1,w_1)\phi(j,r_j') = \phi(1,w_2)\phi(j,r_j') \\
  & \Rightarrow   \phi(1,w_1r_j') = \phi(1,w_2r_j'),
\end{align*}
by \eqref{eqn:hom2}. But since $w_1, w_2 \in H_{1,j}$, both $w_1r_j', w_2r_j' $ represent elements of $H$. Thus, using \eqref{eqRev4} , we have
$$
w_1 r_j' = \psi\phi(1,w_1r_j') = \psi\phi(1,w_2r_j') = w_2 r_j',
$$
and so
$$
w_1 = w_1 r_j' r_j = w_2 r_j' r_j = w_2
$$
in $\ig{\E}$, as required.
\end{proof}

Combining all of these results we arrive at the second main result of this section. 

\begin{thm} \label{thm:reg}
There is an algorithm which takes any finite biordered set $\E$ and computes finite presentations of each of the maximal subgroups $H_e$ $(e \in E)$ of the free idempotent generated semigroup $\ig{\E}$. If each of these finitely presented groups $H_e$ has solvable word problem, then there is an algorithm which given any two words $u, v \in E^*$ decides whether both $u$ and $v$ represent regular elements of $\ig{\E}$ and, if they do, decides whether $u=v$ in $\ig{\E}$. 
\end{thm}

\section{A Rees matrix representation for regular $\D$-classes of $\ig{\E}$}\label{sec_RM}

In this section we continue the investigation of the regular part of $\ig{\E}$ started in the previous sections, concentrating now on finding Rees matrix representations for the completely $0$-simple principal factors of $\ig{\E}$. As with the results above, these representations will be computable when the biorder $\E$ is finite. 

A $\J$-class $C$ in an arbitrary semigroup gives rise to the associated \emph{principal factor} $\overline{C} = C \cup \{0\}$, with multiplication:
\[
x \cdot y = \begin{cases}
xy & \mbox{if $x$, $y$, $xy \in C$} \\
0 & \mbox{otherwise.}
\end{cases}
\] 
It is known that $\overline{C}$ is either a $0$-simple semigroup or a semigroup with zero multiplication (also known as a null semigroup); see \cite[Theorem~3.1.6]{HoBook}. Under some additional finiteness hypotheses, $C$ may happen to be completely $0$-simple, in which case it is isomorphic to a Rees matrix semigroup $\mathcal{M}^0[G;I,J;P]$. Here, $G$ is a group, $I$ and $J$ are index sets,  and $P = (p_{j i})_{j \in J, i \in I}$ is a $J \times I$ matrix with entries from $G \cup \{ 0 \}$ with at least one non-zero entry in every row and column, known as the Rees structure matrix. 
The set of elements is 
\[
\mathcal{M}^0[G;I,J;P]
= \{ (i,g,j) : i \in I, g \in G, j \in J \} \cup \{ 0 \},
\]
and multiplication is defined by 
\[
(i,g,j)(k,h,l) = \begin{cases}
(i, gp_{jk}h, l) & \mbox{if $p_{jk} \neq 0$} \\
0 & \mbox{otherwise},
\end{cases}
\]
and
\[
0 (i,g,j) = (i,g,j)0 = 00 = 0. 
\]
A regular $\D$-class $D$ may or may not be a $\J$-class. When the set of idempotents in $D$ is finite then this certainly is the case \cite[Proposition~3.2.1]{HoBook}, and the corresponding principal factor is completely $0$-simple. 

In what follows we take an arbitrary $\D$-class $D=D_e$ of an idempotent $e$ in $\ig{\E}$, and write down a matrix $P = (p_{j i})_{j \in J, i \in I}$. This matrix turns out to be a Rees structure matrix whenever $D$ is a $\J$-class and the corresponding principal factor is completely $0$-simple. 

Continuing to use the notation and terminology introduced in Section~\ref{sec_WPReg}, recall that the elements $er_j \ (j \in J)$ are representatives of the $\H$-classes $H_{1, j} \ (j \in J)$ in the sense that $er_j \in H_{1,j}$. Next, we construct representatives for the $\H$-classes $H_{i, 1} \ (i \in I)$ in the $\L$-class of $e$ as follows. For every $i \in I$ let $j(i) \in J$ be chosen so that $(i, j(i)) \in K$, i.e. $H_{i, j(i)}$ is a group. From the definitions of $r_j$, $r_j'$ we know that $H_{1, j(i)} r_{j(i)}' = H_{1,1}$. Green's Lemma then implies that $H_{i,j(i)} r_{j(i)}' = H_{i, 1}$. In particular this means that $e_{i,j(i)}r_{j(i)}' \in H_{i, 1}$. 

Now define 
\[
p_{j i} = \begin{cases}
e r_j e_{i, j(i)} r_{j(i)}' &
\mbox{if $(i,j) \in K$} \\
0 &
\mbox{otherwise}. 
\end{cases}
\]
Note that $(i,j) \in K$ means $H_{i, j}$ is a group which together with $er_j \in H_{1, j}$ and $e_{i,j(i)}r_{j(i)}' \in H_{i, 1}$ implies (by the Miller--Clifford theorem) that $p_{j i} \in H_{1, 1}=H$.

It follows from \cite[Theorem~3.2.3]{HoBook} that $P=(p_{j i})_{j \in J, i \in I}$ is a Rees structure matrix for $\overline{D}$ provided $D$ is a $\J$-class and the corresponding principal factor is completely $0$-simple. 
%
Indeed, following the proof of \cite[Theorem~3.2.3]{HoBook} we can write down explicitly an isomorphism between the principal factor $\overline{D}$ and the Rees matrix semigroup $\mathcal{M}^0[H; I, J; P]$. We record the details of this here, since this correspondence will be important in what follows. 

For each $j \in J$ we have chosen and fixed a representative $\overline{er_j}$ of the $\H$-class $H_{1, j}$ of $\ig{\E}$. The bar here signifies the element of $\ig{\E}$ represented by this word. 
Also, for each $i \in I$  we have chosen and fixed a representative $\overline{e_{i, j(i)} r_{j(i)}'}$ of the $\H$-class $H_{i ,1}$. Now, the general theory of regular $\D$-classes tells us that every element of the $\D$-class $D = D_e \subseteq \ig{\E}$ can be written uniquely as
$
\overline{e_{i, j(i)} r_{j(i)}'} \; h \; \overline{er_j} 
$
for some $i \in I$, $j \in J$ and $h \in H$. Then the mapping $\Theta$ defined by
\begin{align}
\overline{e_{i, j(i)} r_{j(i)}'} \; h \; \overline{er_j} \mapsto (i,h,j), 
\quad
0 \mapsto 0, \label{eq_RM_concrete}
\end{align}
is an isomorphism between the principal factor $\overline{D}$ and the Rees matrix semigroup 
$\mathcal{M}^0[H; I, J; P]$ (see \cite[page~74]{HoBook} for a proof that this is an isomorphism). 
The above bar notation will remain in force for the rest of this section, with the aim of
making clear the distinctions and connections between: 
\begin{itemize}
\item[(i)] words over alphabet $E$ that represent elements of the $\D$-class $D$; 
\item[(ii)] the actual elements of the semigroup $\ig{\E}$ from this $\D$-class; and 
\item[(iii)] the corresponding elements of the Rees matrix semigroup $\mathcal{M}^0[H; I, J; P]$. 
\end{itemize}


\begin{rmk}\label{generators-H}
It also turns out (see \cite[Section~3]{GR1}) that the $p_{j i}$ give us an alternative generating set for the maximal subgroup $H$. 
(This remains true even without the assumption that the principal factor $\overline{D}$ is completely $0$-simple.)
In fact, for technical reasons, we prefer the generating set consisting of their inverses, namely
\begin{align}
\{ \overline{e r_{j(i)} e_{i j} r_j'} \; : \; (i,j) \in K \}. \label{H-gen}
\end{align}
\samepage To check they are indeed mutually inverse observe that $\overline{e r_{j(i)} e_{i j} r_j'} \in H_{1, 1}$, 
following the same reasoning as for $p_{ij}$, and then in $\ig{\E}$ we have: 
\begin{align*}
(e r_j e_{i, j(i)} r_{j(i)}')
(e r_{j(i)} e_{i j} r_j' )
& =   
e r_j e_{i, j(i)} r_{j(i)}'
r_{j(i)} e_{i j} r_j' 
& & \mbox{(since $p_{ji} \in H_{1, 1}$)}
\\
& =  
e r_j e_{i, j(i)} e_{i j} r_j' 
& & \mbox{(since $er_ie_{i, j(i)} \in H_{1, j(i)}$)}
\\
& =  
e r_j e_{i j} r_j' 
& & \mbox{(since $e_{i,j(i)} \R e_{ij}$)}
\\
& =  
e r_j r_j' 
& & \mbox{(since $er_j \L e_{i j}$)}
\\
& = e 
& & \mbox{(since $e \in H_{1 ,1}$).}
\end{align*}
\end{rmk}

Inspired by this particular generating set of $H$, in \cite[Theorem~5]{GR1} a presentation for the group $H$ is given in terms of the generators 
\begin{align}
F = \{ f_{ij}\; :\; (i,j)\in K \},  \label{eq_FDef}
\end{align}
where $f_{ij}$ is the formal symbol designated to represent the generator $\overline{e r_{j(i)} e_{i j} r_j'}$ from \eqref{H-gen}.
To make this idea more precise, we define a mapping $\widetilde{\rho} : (F \cup F^{-1})^* \rightarrow (E \cap D)^*$ as the unique (monoid)
homomorphism extending
\begin{align}
f_{i j} 			\mapsto e r_{j(i)} e_{i j} r_j' , \quad \mbox{and} \quad 
f_{i j}^{-1}		 \mapsto e r_j e_{i, j(i)} r_{j(i)}'  \quad 
\mbox{for $(i,j) \in K$.}  \label{tilde-rho}
\end{align}
(The motivation for choosing the notation $\widetilde{\rho}$ will become apparent shortly.) Now we have that $F\widetilde{\rho}$ is a set
of words over $E\cap D$ representing the generating set \eqref{H-gen} of $H$.

The key relations in this presentation are determined by so-called \emph{singular squares}, a notion originally due to Nambooripad \cite{Na}.  
A quadruple $(i,k;j,l)\in I\times I\times J\times J$ is a \emph{square}
if $(i,j),(i,l),(k,j),(k,l)\in K$.
It is a \emph{singular square} if, in addition, there exists an idempotent $f\in E$ such that
one of the following two dual sets of conditions holds:
\begin{align}
\label{liseq2}
& fe_{ij}=e_{ij},\ fe_{kj}=e_{kj},\ e_{ij}f=e_{il},\ e_{kj}f=e_{kl},\ \mbox{or}\\
\label{liseq3}
& e_{ij}f=e_{ij},\ e_{il}f=e_{il},\ fe_{ij}=e_{kj},\ fe_{il}=e_{kl}.
\end{align}
We will say that $f$ \emph{singularises} the square.
Let $\Sigma_{LR}$ (respectively $\Sigma_{UD}$) be the set of all singular squares for 
which condition \eqref{liseq2}
(resp. \eqref{liseq3}) holds, and
let $\Sigma=\Sigma_{LR}\cup\Sigma_{UD}$, the set of all singular squares.
We call the members of $\Sigma_{LR}$ the \emph{left-right} singular squares, and those of $\Sigma_{UD}$ the \emph{up-down} singular squares.

Combining the above observations with the presentation for the group $H$ obtained in \cite[Theorem~5]{GR1} yields the following result. 

\begin{thm}\label{thm_presentation}
Let $S$ be a semigroup with a non-empty set of idempotents $E$,
let $\ig{\E}$ be the corresponding free idempotent generated semigroup,
let $e\in E$ be arbitrary, and let $H$ be the maximal subgroup of $e$ in $\ig{\E}$.
With the rest of notation as introduced throughout this section, a presentation for $H$ is given by
\begin{align}
\label{rels51}
&\langle f_{ij}\ ((i,j)\in K) &&|&&
f_{ij}=f_{il} && ((i,j),(i,k)\in K,\ r_je_{il}=r_{j\cdot e_{il}}),
\\
\label{rels52}
& && &&f_{i,j(i)}=1 && (i\in I),
\\
\label{rels53}
& && &&f_{ij}^{-1}f_{il}=f_{kj}^{-1}f_{kl} && ((i,k;j,l)\in\Sigma) \rangle.
\end{align}
Furthermore, if $D_e = J_e$ in $\ig{\E}$, and the corresponding principal factor $\overline{D}_e$ is completely $0$-simple, then 
\begin{align}\label{lab_RMS}
\overline{D}_e \cong \mathcal{M}^0[H; I, J; P]
\end{align}
where 
\[
p_{j i} = \begin{cases}
f_{i j}^{-1} & (i,j) \in K \\
0 & \mbox{otherwise}. 
\end{cases}
\]
\end{thm}


Suppose now that $e$ is an idempotent such that 
$D_e = J_e$ in $\ig{\E}$ and the corresponding principal factor $\overline{D}_e$ is completely $0$-simple. 
Since $\ig{\E}$ is idempotent generated it follows by a result of  Fitz-Gerald \cite{FG} that every element of $D_e$ can be written as a product of idempotents from $D_e$. Given such a word $w \in (E \cap D)^*$ we now describe how to obtain a triple $(i, \gamma, j)$ where $\gamma \in (F \cup F^{-1})^*$ representing the same element of $D_e$, giving an explicit description of the isomorphism \eqref{lab_RMS}
on the level of words.

Define 
\begin{align*}
\widetilde{\pi} : (E \cap D)^* & \rightarrow  (F \cup F^{-1})^* \\
e_{i_1 j_1}e_{i_2 j_2} \ldots e_{i_m j_m}
& \mapsto 
f_{i_1 j_1}f_{i_2 j_1}^{-1}
f_{i_2 j_2}f_{i_3 j_2}^{-1}
\ldots
f_{i_{m-1} j_{m-1}}f_{i_{m} j_{m-1}}^{-1}
f_{i_m j_m}.
\end{align*}
Noting that under the isomorphism $\Theta$ defined in \eqref{eq_RM_concrete} 
the idempotent 
$e_{i j}$ corresponds 
to the idempotent $(i, f_{i j}, j)$ in $\mathcal{M}^0[H; I, J; P]$, it follows that when 
$$w \equiv e_{i_1 j_1}e_{i_2 j_2} \ldots e_{i_m j_m}$$
represents an element of $D$ it corresponds to the triple $(i_1, \widetilde{\pi}(w), j_m)$. Hence, the mapping
\begin{align}
\pi: \{
\mbox{
$w \in (E \cap D)^*$: $w$ represents an element of $D$
}  \} & \rightarrow 
I \times (F \cup F^{-1})^* \times J  \notag \\
w\equiv e_{i_1 j_1}e_{i_2 j_2} \ldots e_{i_m j_m} & \mapsto  (i_1, \widetilde{\pi}(w), j_m) \label{pidef}
\end{align}
takes any word $w$ over $E \cap D$ which represents an element of $D$ and returns  a triple $\pi(w) = (i,u,j) \in I \times (F \cup F^{-1})^* \times J$  with the property that the element of $\ig{\E}$ represented by the word $w$ corresponds to the element $(i,u,j)$ of the Rees matrix semigroup $\mathcal{M}^0[H; I, J; P]$ via the isomorphism $\Theta$ defined in \eqref{eq_RM_concrete}, where the last $u$ is interpreted as an element of $H$. 

Going the other way, we define a mapping
\begin{align}
\rho: 
I \times (F \cup F^{-1})^* \times J 
& \rightarrow  (E \cap D)^*  \notag \\
(i,w,j) & \mapsto 
e_{i ,j(i)} r_{j(i)}' \widetilde{\rho}(w) e_{1 1} r_j, \label{rhodef}
\end{align}
where $\widetilde{\rho}$ is the homomorphism $(F \cup F^{-1})^* \rightarrow (E \cap D)^*$ already defined in \eqref{tilde-rho}.
Now we have that for every word $w \in (F \cup F^{-1})^*$ the word $\widetilde{\rho}(w)\in (E \cap D)^\ast$ represents the same element of $H$
(considered as a maximal subgroup of $\ig{\E}$) as $w$ does (with respect to the presentation given in Theorem \ref{thm_presentation}). 
It follows that 
$\rho(i,w,j) = 
e_{i ,j(i)} r_{j(i)}' \widetilde{\rho}(w) e_{1 1} r_j$ 
is a word over $E \cap D$ representing the element 
\begin{align}
\overline{e_{i, j(i)} r_{j(i)}'} \; h \; \overline{er_j}
\label{eq_ddagger}
\end{align}
where $h \in H$ is the element of $H$ represented by the word $w$. 
This element of $\ig{\E}$ in turn corresponds to the element $(i,h,j)$ of the Rees matrix semigroup 
$\mathcal{M}^0[H; I, J; P]$ via the isomorphism $\Theta$ defined in \eqref{eq_RM_concrete}. 
As a consequence of this correspondence, and the definition of multiplication in $\mathcal{M}^0[H; I, J; P]$,
it follows that for all $i, k \in I$, $j, l \in J$ and $u, v \in (F \cup F^{-1})^*$ we have
\begin{align}
\rho(i,u,j)\rho(k,v,l) =
\rho(i, uf_{k j}^{-1}v, l)
\label{eq_rhoformula}
\end{align}
in $\ig{\E}$. 
Also, if $w_1, w_2 \in (F \cup F^{-1})^*$ both represent the same element $h$ of $H$ then both the words $\rho(i,w_1,j)$ and $\rho(i,w_2,j)$ represent the same element of $\ig{\E}$, namely the element \eqref{eq_ddagger}. 
The following result records the relationship between the mappings $\pi$ and $\rho$. 
\begin{lem}
\label{rho-computable}
The mappings $\pi$ and $\rho$ induce mutually inverse isomorphisms between the principal factor $\overline{D}_e$ and the Rees matrix semigroup 
$\mathcal{M}^0[H; I, J; P]$. Moreover, when $\E$ is finite both of these mappings are effectively computable. 
\end{lem}
\begin{proof}
It is a straightforward consequence of the definitions and discussion above that
on the level of semigroups (as opposed to words) the mappings $\pi$ and $\rho$ induce the isomorphism $\Theta$ (defined in \eqref{eq_RM_concrete}) and its inverse respectively, between the principal factor $\overline{D}_e$ and the Rees matrix semigroup 
$\mathcal{M}^0[H; I, J; P]$. By Lemma~\ref{lem_zero_c} and Theorem~\ref{thm:testing_reg_c} we can decide if a word represents an element of $D$, and the indices $i,j$ of the idempotents $e_{i j}$ can be computed, and hence the mapping $\pi$ can be computed using formula \eqref{pidef}. The mapping $\rho$ is computable using formula \eqref{rhodef} since in the discussion preceding the statement of Theorem~\ref{thm:prescomp} we saw that there is an algorithm which computes a Schreier system $r_j, r_j' \; (j \in J)$ from the finite biorder $\E$.  
%
%
%
%
\end{proof}

For the remainder of the article, extensive use will be made of the mapping $\rho$. We highlight here the two key properties of $\rho$ that will be used throughout: 
\begin{enumerate}
\item[(i)] The word $\rho(i,w,j) \in (E \cap D)^*$ represents the element of $\ig{\E}$ that corresponds (i.e. maps onto under the isomorphism $\Theta$ given in \eqref{eq_RM_concrete}) to the triple $(i,w,j)$ of the Rees matrix semigroup $\mathcal{M}^0[H; I, J; P]$ (with $w$ interpreted as an element of $H$).
\item[(ii)] When $\E$ is finite the mapping $\rho$ is effectively computable.
\end{enumerate}
In the arguments in Section~\ref{sec:undec} below, $\rho(i,w,j)$ will sometimes be used simply as convenient notation for a word over $E$ representing the triple $(i,w,j)$, while at other points in the argument it will be of crucial importance that the mapping $\rho$ is effectively computable.

This completes our discussion of 
the regular part of $\ig{\E}$ and decidability properties pertaining to it. 
In the rest of the paper we turn our attention to the non-regular part of $\ig{\E}$. 

%
%
%
%
%
%
%
%
%
%
%
%
\section{Undecidability of the word problem in general}\label{sec_undec1}
\label{secWPgen}

The word problem for a free idempotent generated semigroup $\ig{\E}$ of a finite biordered set $\E$ is undecidable in general. However, until now, the only known examples contained maximal subgroups with undecidable word problem. This led naturally to the question of whether this was the only barrier to undecidability of the word problem in $\ig{\E}$, 
specifically: 

\

\noindent \textbf{Question:} If $\E$ is a finite biordered set, and every maximal subgroup of $\ig{\E}$ has decidable word problem, does it follow that $\ig{\E}$ has decidable word problem?   

\

\noindent The results above show that under these assumptions many properties are decidable in $\ig{\E}$, in particular regularity, and the word problem for regular words. The rest of the article will be devoted to 
showing that the answer to the above question is no. In this section we outline our general approach to the problem, and then 
in subsequent sections 
we give full details of the construction, and proofs of the results needed to establish undecidability. 

The key idea of the construction is to relate the word problem in $\ig{\E}$ to the membership problem for finitely generated subgroups of finitely presented groups. The construction takes a finitely presented group $G=\pre{A}{\mathcal{R}}$ and a finitely generated subgroup $H$ of $G$, where we suppose that: 
\begin{itemize}
\item[(G1)] every relation from $\mathcal{R}$ has the form $ab=c$ for some $a,b,c\in A$;
\item[(G2)] $H$ is specified by the set of generators $B$ such that $B\subseteq A$. Furthermore, we assume that $B^{-1}=B$, so that every element of $H$ can be expressed as a (monoid) word over $B$.
\end{itemize}
%
%
To see that given any finitely presented group $G$, and finitely generated subgroup $H$ of $G$, a finite presentation $\pre{A}{\mathcal{R}}$ for 
$G$ satisfying (G1), (G2) exists, one can proceed as follows: 
\begin{enumerate}
\item Pick a finite generating set $B_1$ of $H$.
\item Extend $B_1$ to a finite generating set $A_1$ of $G$.
\item Close $A_1$ under inversion, and add a letter $e$ representing the identity element to it, to obtain $A_2=A_1\cup A_1^{-1}\cup\{e\}$.
\item Write down a finite presentation for $G$ over generators $A_1$, in which every relation has the form $u=e$ where $u$ is a positive
      word, and which contains the relations $ae=a$, $ea=a$   ($a\in A_1$) and $ab=e$ (for any pair of mutually inverse generators 
			$a,b\in A_1$).
\item Expand $A_1$ to a new generating set $A$, by adding to it symbols representing all non-empty prefixes of all words $u$ featuring in 
      the relations $u=e$.
\item Write down all the relations of the form $ab=c$ over this new generating set that hold in $G$, and check that this presentation has 
      the desired properties.
\end{enumerate}

From this data we shall construct a finite band which we denote by $B_{G,H}$,
and the corresponding biorder by $\B_{G,H}$.
The band $B_{G,H}$ will have five $\D$-classes whose dimensions are determined by the sizes of the finite sets $A$, $B$ and $\mathcal{R}$, the minimal $\D$-class is a zero element, and the $\J$-order on these $\D$-classes is illustrated in Figure~\ref{fig_BGH}. 
Full details of the construction of the band $B_{G,H}$ will be given in Section~\ref{sec_BGH} below. 

Let $\ig{\B_{G,H}}$ denote the free idempotent generated semigroup arising from the band $B_{G,H}$. After giving details of the construction, and making some observations about the structure of $B_{G,H}$, we shall then go on to prove the following results about $\ig{\B_{G,H}}$. 

\begin{thm}\label{thm_SubG}
Every non-trivial maximal subgroup of $\ig{\B_{G,H}}$ is isomorphic to the group $G$. 
\end{thm}

\begin{thm}\label{thm_Memb}
If $\ig{\B_{G,H}}$ has decidable world problem then the membership problem for $H$ in $G$ is decidable. 
\end{thm}

Theorem \ref{thm_SubG} is proved in Proposition \ref{maxsubgps}, and Theorem 
\ref{thm_Memb} is proved at the very end of Section \ref{sec:undec}.
Combining these results leads us to our main theorem for $\ig{\B_{G,H}}$.

\begin{thm}\label{thm_BGHMain}
Let $G$ be a finitely presented group with decidable word problem and let $H$ be a finitely generated subgroup of $G$ with undecidable membership problem. Then the free idempotent generated semigroup $\ig{\B_{G,H}}$ over the finite band $B_{G,H}$ has the following properties:  
\begin{itemize}
\item[(i)] every non-trivial maximal subgroup of $\ig{\B_{G,H}}$ is isomorphic to $G$, and so every maximal subgroup of $\ig{\B_{G,H}}$ has decidable word problem, while
\item[(ii)] the semigroup $\ig{\B_{G,H}}$ has undecidable word problem.  
\end{itemize}
\end{thm}

It is known from combinatorial group theory that pairs of groups $H \leq G$, satisfying the hypotheses of Theorem~\ref{thm_BGHMain} do exist. 
The first such example, based on fibre products, was exhibited by Mihailova \cite{Mihailova}, and proceeds as follows.
Let $\Gamma$ be a group and $\theta : \Gamma\rightarrow \Delta$ a group homomorphism.
The associated \emph{fibre product} is
\[
\Pi_{\Gamma,\theta} =\{ (g_1,g_2)\in \Gamma\times\Gamma : g_1\theta=g_2\theta\}\leq \Gamma\times \Gamma.
\]
If $\Gamma$ is generated by a set $A$ and if $\ker\theta$ is generated as a normal subgroup
by a set $R$, then it can be shown that $\Pi_{\Gamma,\theta}$ is generated by 
$\{ (a,a): a\in A\}\cup \{ (r,1): r\in R\}$; see \cite[Lemma 4.1]{Miller}.
If we choose $\Delta=\langle A|R\rangle$ to be a finitely presented group with undecidable word problem, and $\Gamma$ to be the free group on $A$, it follows that
$\Pi_{\Gamma,\theta}$ is a finitely generated subgroup of the finitely presented group $\Gamma\times\Gamma$ with undecidable membership problem. To summarise:

\begin{thm}[Mihailova (1958)]
Let $\Gamma$ be a finitely generated free group of rank at least $2$. Then $G = \Gamma \times \Gamma$ is a group with decidable word problem, and $G$ has a finitely generated subgroup $H$ such that the membership problem for $H$ in $G$ is undecidable. 
\end{thm}

Combining this with Theorem~\ref{thm_BGHMain} completes the proof of our main result. 

\begin{thm} \label{thm:undec}
There exists a finite band $S$ whose biordered set $\E=\E(S)$ has the following properties:
\begin{itemize}
\item[(i)] All maximal subgroups of $\ig{\E}$ have decidable word problem.
\item[(ii)] The word problem for $\ig{\E}$ is undecidable.
\end{itemize}
\end{thm}

\section{The $B_{G,H}$ construction}\label{sec_BGH}
\noindent We begin by describing a general construction. 
Consider any triple $S$, $V$, $U$ where 
\begin{itemize}
\item $S$ is a semigroup, 
\item $U$ is an ideal of $S$, and
\item $V$ is a subsemigroup of $S$.  
\end{itemize}
Let $S^{(1)}$, $S^{(2)}$ be copies of $S$, all three pairwise disjoint, under isomorphisms $s \mapsto s^{(i)}$ ($i=1,2$). Let
\[
W(S) = 
S \cup S^{(1)} \cup S^{(2)} \cup \{0\}
\]
be the semigroup with $0$ extending the multiplication on $S$, $S^{(1)}$, $S^{(2)}$ via:
\[
st^{(i)} = s^{(i)}t = (st)^{(i)}, \quad
s^{(1)} t^{(2)} = s^{(2)} t^{(1)} = 0, 
\]
for $s, t \in S$ and $i=1,2$. That this is indeed a semigroup can be easily checked directly, or by noting that this is a special instance of the Clifford construction \cite[Section~4.2]{HoBook}, with the ingredients arranged in a diamond semilattice as in Figure~\ref{fig_pictorial}. Let 
\[
T(S, V, U) = V \cup U^{(1)} \cup U^{(2)} \cup \{ 0 \}. 
\]
It is easy to verify that $T(S,V,U)$ is a subsemigroup of $W(S)$. Its gross structure, and the way that it embeds into $W(S)$, are
also illustrated in Figure~\ref{fig_pictorial}. 

%
%

\begin{figure}[t]
\begin{center}
\scalebox{1}
{
\begin{tikzpicture}
[scale=1, 
Srectangle/.style ={draw, rectangle, thick, rounded corners, minimum height=8em, minimum width=6em},
Vrectangle/.style ={draw, rectangle, thick, rounded corners, minimum height=6em, minimum width=3.8em},
Irectangle/.style ={draw, rectangle, thick, rounded corners, minimum height=2em, minimum width=6em},
FSrectangle/.style ={rectangle, draw=black, thick, fill=gray!20, rounded corners, minimum height=8em, minimum width=6em},
FVrectangle/.style ={rectangle, draw=black, thick, fill=gray!20, rounded corners, minimum height=6em, minimum width=3.8em},
FIrectangle/.style ={rectangle, draw=black, thick, fill=gray!20, rounded corners, minimum height=2em, minimum width=6em}],
\tikzstyle{vertex}=[circle,draw=black, fill=black, inner sep = 0.3mm]
%
%
\node (S) at (0em,0em) [Srectangle] {};
\node at (-1.1em,-1em) [FVrectangle] {};
\node at (0em,-3em) [Irectangle] {};
\node at (4em,4em) {$S$};
\node at (-2em,1em) {$V$};
\node at (2em,-3em) {$U$};
%
%
\node (S'') at (8em,-9em) [Srectangle] {};
\node at (8em,-12em) [FIrectangle] {};
\node at (6.9em,-10em) [Vrectangle] {};
\node at (12.5em,-5em) {$S^{(2)}$};
\node at (6.1em,-8em) {$V^{(2)}$};
\node at (10em,-12em) {$U^{(2)}$};
%
%
\node (S') at (-8em,-9em) [Srectangle] {};
\node at (-8em,-12em) [FIrectangle] {};
\node at (-9.1em,-10em) [Vrectangle] {};
\node at (-12em,-5em) {$S^{(1)}$};
\node at (-9.8em,-8em) {$V^{(1)}$};
\node at (-6em,-12em) {$U^{(1)}$};
%
%
\node (Z) at (0,-0.75-6) [circle,draw=white!20,fill=gray!20] {$\{ 0 \}$};
%
%
\path[-stealth] (-8em,-14em) edge node [below] {$0$} (Z);
\path[-stealth] (8em,-14em) edge node [below] {$0$} (Z);
\path[-stealth] (-4em,0em) edge node [above] {$\mathrm{id}$ \ \ } (-8em,-4em);
\path[-stealth] (4em,0em) edge node [above] { \ \ $\mathrm{id}$} (8em,-4em);
\end{tikzpicture}
}
\end{center}
\caption{
A pictorial representation of the structure of the semigroup $W(S)$. The subsemigroup $T(S, V, U) = V \cup U^{(1)} \cup U^{(2)} \cup \{ 0 \}$ of $W(S)$ is shaded in grey. 
}\label{fig_pictorial}
\end{figure}
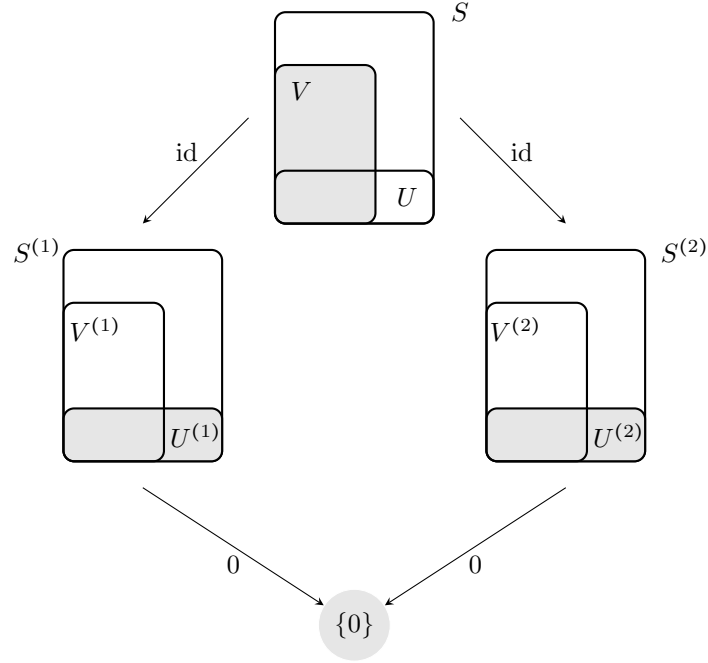

Let $G=\pre{A}{\mathcal{R}}$ be a finitely presented group and let $H$ be a finitely generated subgroup of $G$. 
Again, as in Section \ref{secWPgen}, we assume without loss of generality that the presentation 
$\pre{A}{\mathcal{R}}$ is finite and satisfies Conditions (G1), (G2).
We are now going to describe a finite band $B_{G,H}$. 
It will be obtained using the $T(S,V,U)$ construction above where $S=B_G$ is taken to be the band constructed in \cite{DR}, $U$ will be the unique minimal ideal of $B_G$, and $V$ will be a certain subsemigroup of $B_G$ which will depend on the choice of the subgroup $H$. The details are as follows. 


Define $A_1=A\cup\{1\}$ and $B_1=B\cup\{1\}$; furthermore, let $\ol{A}_1=\{\ol{a}:\ a\in A_1\}$ be a copy of $A_1$, and let
$\infty$ be a symbol not in $A_1$. Set
$$I=A_1\cup\ol{A}_1,\quad J=A_1\cup\{\infty\}.$$
Now consider the direct product $\T=\T_I^{(l)}\times \T_J^{(r)},$ where $\T_I^{(l)}$ (respectively $\T_J^{(r)}$) is the
semigroup of all mappings $I\rightarrow I$ (resp.\ $J\rightarrow J$) written on the left (resp. right). The semigroup
$\T$ has a unique minimal ideal $K_G$ consisting of all $(\sigma,\tau)$ with both $\sigma$ and $\tau$ constant. This
ideal is naturally isomorphic to the rectangular band $I\times J$, and we identify the two. Following \cite{DR} we have
$B_G=K_G\cup L_G\subseteq \T_I^{(l)}\times \T_J^{(r)}$, where $L_G$ is a left zero band which we now describe.

For each $(\sigma,\tau)\in L_G$ we shall have
$$
\sigma^2=\sigma,\ \ \tau^2=\tau,\ \ \ker(\sigma)=\{A_1,\ol{A}_1\},\ \ \im(\tau)=A_1.
$$
Therefore, each $(\sigma,\tau)$ will be uniquely determined by $\im(\sigma)$ which must be a two-element set that is a
cross-section of $\{A_1,\ol{A}_1\}$, and the value $(\infty)\tau\in A_1$. 
These idempotents are defined in Table~\ref{table_table}. Note that $e_1=e_{\ol{1}}$.
\begin{table}[b]
\begin{center}
\begin{tabular}{|c|c|c|c|}
\hline
 \rule[-4pt]{0pt}{14pt} \textbf{Notation} &  \textbf{Indexing} & {\boldmath $\im(\sigma)$} & {\boldmath $(\infty)\tau$}
\\
\hline \rule[-5pt]{0pt}{15pt} $e_a=(\sigma_a\,,\tau_a)$ & $a\in A_1$ & $\{1,\ol{a}\}$ & $a$
\\
\hline \rule[-5pt]{0pt}{15pt} $e_{\ol{a}}=(\sigma_{\ol{a}}\,,\tau_{\ol{a}})$ & $\ol{a}\in \ol{A}_1$ & $\{a,\ol{a}\}$ &
$1$
\\
\hline \rule[-5pt]{0pt}{15pt} $e_{\rrel}=(\sigma_{\rrel}\,,\tau_{\rrel})$ & $\rrel=(ab,c)\in \mathcal{R} $ & $\{b,\ol{c}\}$ & $a$
\\
\hline
\end{tabular}

\bigskip

\end{center}
\caption{The elements of the $\D$-class $L_G$ of the band $B_{G,H}$.}\label{table_table}
\end{table}
From the previous definition it is easy to see that for any $e=(\sigma,\tau)\in L_G$ we have that $\sigma|_{A_1}$ and
$\tau|_{B_1}$ are transformations of $A_1$ and $B_1$, respectively (more precisely, a constant map and the identity).
Therefore, if we take $K_H=A_1\times B_1$ as a rectangular subband of $K_G$, $\ol{B}_{G,H}=K_H\cup L_G$ is a subband of
$B_G$. So we have
\[
K_H \cup L_G \leq K_G \cup L_G = B_G. 
\]
Also note that $K_G$ is an ideal of $K_G \cup L_G = B_G$. Therefore all of the conditions are satisfied to apply our general construction above and we define the semigroup
\[
B_{G,H}=
T(B_G, K_H \cup L_G, K_G)=
T(K_G \cup L_G, K_H \cup L_G, K_G)
\]
We now make some observations about the semigroup $B_{G,H}$. 
\begin{itemize}
\item $B_{G,H}$ is a finite band with five $\D$-classes $L_G$, $K_H$, $K_G^{(1)}$, $K_G^{(1)}$, and $\{0\}$ with their $\J$-order illustrated in Figure~\ref{fig_BGH}. Compare with Figure~\ref{fig_pictorial}, noting that $K_H$ corresponds to the intersection $V \cap U$. 
\item We have isomorphisms 
\[
L_G \cup K_G^{(1)} \cong
L_G \cup K_G^{(2)} \cong
L_G \cup K_G = B_G,
\]
via the obvious bijections, $x \mapsto x$ for $x \in L_G$ and $x \mapsto x^{(i)}$ for $x \in K_G$. 
\item The elements of $L_G$ are 
\[
e_a \ (a \in A_1), \quad
e_{\overline{a}} \ (\overline{a} \in \overline{A}_1), \quad
e_r \ (r \in \mathcal{R}),
\]
which are defined in Table~\ref{table_table}. 
\item The elements of $K_H$ are pairs $(a,b)$, $a \in A_1$, $b \in B_1$ which can be identified with pairs of constant maps. 
\item The elements of $K_G^{(i)}$ are pairs 
\[
(a_1^{(i)}, a_2^{(i)}), \quad
a_1^{(i)} \in A_1^{(i)} \cup \overline{A}_1^{(i)}, \
a_2^{(i)} \in A_1^{(i)} \cup \{ \infty^{(i)} \}. 
\]
\end{itemize}
%
\begin{figure}
\begin{center}
\scalebox{1}
{
\begin{tikzpicture}
[scale=1, 
LGRectangle/.style={draw, thick, rectangle, minimum height=3cm, minimum width=0.5cm},
KHRectangle/.style={draw, thick, rectangle, minimum height=2cm, minimum width=1.5cm},
KGRectangle/.style={draw, thick, rectangle, minimum height=4cm, minimum width=2.5cm}],
\tikzstyle{vertex}=[circle,draw=black, fill=black, inner sep = 0.3mm]
%
%
\node at (0,9) [LGRectangle] {};
\draw[thick] (-0.25,9.5)--(0.25,9.5);
\draw[thick] (-0.25,8.5)--(0.25,8.5);
\node (A1bar) at (-0.6,9) {$\overline{A}_1$};
\node (A1) at (-0.6,10) {$A_1$};
\node (LG) at (1,10+0.25) {$L_G$};
\node (R) at (-0.6,8) {$R$};
\draw[thick] (0,7.4)--(0,7.1);
%
\node at (0,6) [KHRectangle] {};
\node (KH) at (1.5,6.5+0.25) {$K_H$};
\node (A1KH) at (-1.125,6) {$A_1$};
\node (B1KH) at (0,4.75-0.125) {$B_1$};
\draw[thick] (1,5)--(1+0.3,5-0.3);
\draw[thick] (-1,5)--(-1-0.3,5-0.3);
%
\node at (-2,2.5) [KGRectangle] {};
\draw[thick] (-2-1.25,2.5)--(-2+1.25,2.5);
\draw[thick] (-1.25,4.5)--(-1.25,0.5);
\node (KGDash) at (-3.25,5) {$K_G^{(1)}$};
\node (A1DLeft) at (-2-1.625,2.5+1) {$A_1^{(1)}$};
\node (A1DBLeft) at (-2-1.625,2.5-1) {$\overline{A}_1^{(1)}$};
\node (A1DLeft2) at (-2-0.5+0.25,2.5-2.25-0.125) {$A_1^{(1)}$};
\node (InftyDLeft) at (-2-0.5+0.25+1.3+0.1,2.5-2.25-0.125) {$\infty^{(1)}$};
%
\node at (2,2.5) [KGRectangle] {};
\draw[thick] (2-1.25,2.5)--(2+1.25,2.5);
\draw[thick] (-1.25+4,4.5)--(-1.25+4,0.5);
\node (KGDashD) at (-3.25+6.5,5) {$K_G^{(2)}$};
\node (DA1DLeft) at (-2-1.625+4,2.5+1) {$A_1^{(2)}$};
\node (DA1DBLeft) at (-2-1.625+4,2.5-1) {$\overline{A}_1^{(2)}$};
\node (DA1DLeft2) at (-2-0.5+0.25+4,2.5-2.25-0.125) {$A_1^{(2)}$};
\node (DInftyDLeft) at (-2-0.5+0.25+1.3+4+0.1,2.5-2.25-0.125) {$\infty^{(2)}$};

\draw[thick] (1+0.3,5-5.3)--(1,5-0.3-5.3);
\draw[thick] (-1-0.3,5-5.3)--(-1,5-0.3-5.3);
%
\node (Z) at (0,-0.75) {$\{ 0 \}$};
\end{tikzpicture}
}
\end{center}
\caption{
An illustration of the structure of the finite band $B_{G,H}$. 
}\label{fig_BGH}
\end{figure}
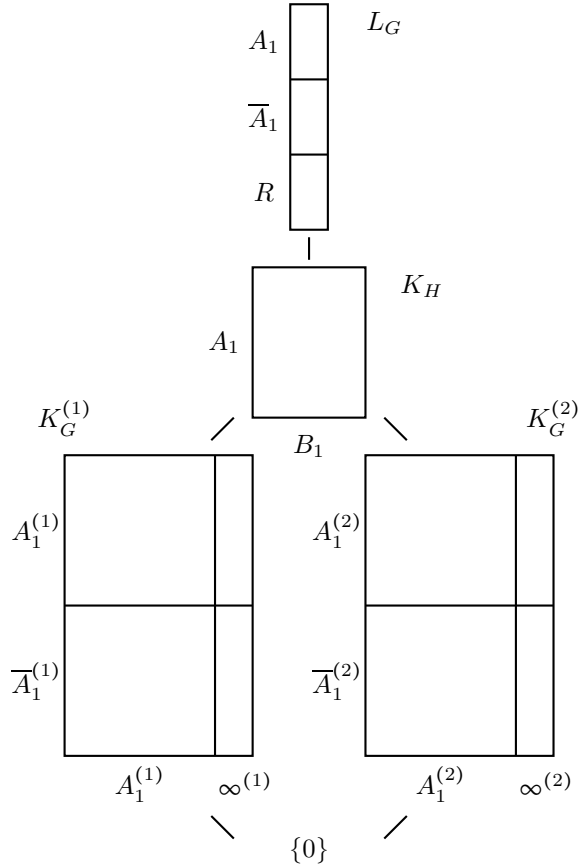
It follows from the definition of the construction $T(S,V,U)$ that, with the other notation introduced above, the elements 
$(\sigma, \tau) \in L_G$, $(a,b) \in K_H$, $(a_1^{(i)}, a_2^{(i)}) \in K_G^{(i)}$
%
%
multiply in the following way: 
\begin{align}
(\sigma,\tau)(a,b) & = (\sigma(a),b) &  
(a,b)(\sigma,\tau) & = (a, b\tau) \notag \\
(\sigma,\tau)(a_1^{(i)}, a_2^{(i)}) & =
(\sigma(a_1)^{(i)}, a_2^{(i)}) & 
(a_1^{(i)}, a_2^{(i)})(\sigma,\tau) & = 
(a_1^{(i)}, (a_2\tau)^{(i)}) \label{eqn_multiplication} \\
(a,b)(a_1^{(i)}, a_2^{(i)}) & =
(a^{(i)}, a_2^{(i)}) & 
(a_1^{(i)}, a_2^{(i)})(a,b) & = 
(a_1^{(i)}, b^{(i)}), \notag
\end{align}
while for $(a_1^{(1)},a_2^{(1)}) \in K_G^{(1)}$ and $(a_3^{(2)},a_4^{(2)}) \in K_G^{(2)}$ we have
\[
(a_1^{(1)},a_2^{(1)})(a_3^{(2)},a_4^{(2)}) = 
(a_3^{(2)},a_4^{(2)})(a_1^{(1)},a_2^{(1)}) = 0. 
\]
Within a single $\D$-class the multiplication is the usual rectangular band multiplication. 
An illustration of the band $B_{G,H}$ in given in Figure~\ref{fig_BGH}. It shows the five $\D$-classes, the indexing sets of each of these $\D$-classes, and their $\D$-class poset ordering.  

\section{The maximal subgroups of $\ig{\B_{G,H}}$}\label{sec_maxsub}

We shall now describe the maximal subgroups of $\ig{\B_{G,H}}$ showing that they are all either trivial or isomorphic to $G$, thus establishing Theorem~\ref{thm_SubG}. 

As a consequence of the property (IG3) from \cite{GR1} and results of \cite{BMM1}, the regular $\D$-classes of $\ig{\B_{G,H}}$
will be in a natural bijective correspondence with the $\D$-classes of $B_{G,H}$;
to emphasise this correspondence, we shall denote the
$\D$-classes in $\ig{\B_{G,H}}$ by
$\overline{L}_G$, $\overline{K}_H$, $\overline{K}_G^{(1)}$, $\overline{K}_G^{(2)}$ and $\overline{\{0\}}$.
The maximal subgroups of $\ig{\B_{G,H}}$ will be located
within these five regular $\D$-classes. Furthermore, since $B_{G,H}$ is a band---so that each $\H$-class 
consists of a single idempotent---it follows that each regular $\D$-class of $\ig{\B_{G,H}}$ is the union of its maximal subgroups. 

\begin{pro}
\label{maxsubgps}
The maximal subgroups of $\ig{\B_{G,H}}$ contained in $\overline{L}_G$, $\overline{K}_H$ or $\overline{\{0\}}$ are trivial, while those contained in $\overline{K}_G^{(1)}$ or $\overline{K}_G^{(2)}$
are isomorphic to $G$
\end{pro}

\begin{proof}
All the statements follow from Theorem \ref{thm_presentation}.
Since $L_G$ is a left zero semigroup, it follows that for every generator $f$ of
$\overline{L}_G$ we have $f=1$ (defining relations \eqref{rels52}), and so the group is trivial.
Recall that the $\D$-class $K_H$ of $B_{G,H}$ is an $A_1\times B_1$ rectangular band.
Consider the action of $e_{\overline{a}}=(\sigma,\tau)\in L_G$ on $K_H$: $\sigma$ acts as the constant map on $A_1$ with value $a$, while $\tau$ acts as the identity map on $B_1$. 
It follows that $e_{\overline{a}}$ singularises every square of the form $(a,a_1; b,b_1)$ with $a_1\in A_1$ and $b,b_1\in B$.
By varying $a\in A_1$ we see that \emph{all} the squares are singular, and hence the presentation from Theorem \ref{thm_presentation}  again defines the trivial group.
The $\D$-class $\{0\}$ is trivial, and hence its counterpart $\overline{\{0\}}$ is trivial as well. 


Let us now consider the $\D$-class $\overline{K}_G^{(i)}$, $i=1,2$. The only idempotents of $B_{G,H}$ acting on $K_G$
are those from $L_G\cup K_H$. However, the idempotents from $K_H$ act by constant maps, and hence induce no non-trivial singular squares. It follows that the maximal subgroups of $\ig{\B_{G,H}}$ inside $\overline{K}_G^{(i)}$
are isomorphic to the maximal subgroups of the free idempotent generated semigroup over the band $L_G\cup K_G^{(i)}$ in the $\D$-class corresponding to $K_G^{(i)}$.
The band $L_G\cup K_G^{(i)}$ is isomorphic to that constructed in \cite{DR}, where it is proved that the maximal subgroups in question are isomorphic to $G$,
as required.
\end{proof}

It will be helpful for our subsequent exposition to outline briefly how the above mentioned argument from \cite{DR} proceeds.
It is a Tietze transformations argument, starting from the presentation given by Theorem \ref{thm_presentation}
in terms of the generators $f_{ij}$, $i\in I=A_1\cup\overline{A}_1$, $ j\in J=A_1\cup\{\infty\}$.
A systematic analysis of singular squares induced by the idempotents $e_a$ and $e_{\overline{a}}$ from $L_G$  shows that large collections of the above generators are equal to each other.
More specifically, it is proved that all the generators $f_{ax}$, with
$(a,x)\in (A_1\times A_1) \cup \{ (1,\infty)\} \cup (\overline{A}_1\times\{1\}) \cup \{  (\overline{1}, \infty)  \}$, are equal to $1$,
and that for each $a\in A$ the generators $f_{a,\infty}$, $f_{\overline{x}, a}$ ($x\in A_1$) and $f_{\overline{a},\infty}$
are all equal.
Renaming formally all the generators in this latter group as $a$, and considering the remaining singular squares induced by the idempotents $e_{\rrel} \in L_G$ $(\rrel \in \mathcal{R})$ yields the original presentation 
$\pre{A}{\mathcal{R}}$ of $G$. 
The relationship between the original generators $A$ and the generators $f_{ij}$ is summarised in the table in Figure~\ref{A-vs-fij},
and will be referred to throughout the technical argument in the following section.

\begin{figure}
\begin{center}
\begin{tikzpicture}
  \matrix (m) [matrix of math nodes,row sep=0.5em,column sep=0.1em,minimum width=2em]
  {
& 1 & a_1 & a_2 & \ldots & \infty
\\ 
1 & 1&1&1& \ldots & 1
\\
a_1 &1&1&1&& a_1
\\
a_2 & 1&1&1& \ldots & a_2
\\
\vdots & \vdots && \vdots & \ddots & \vdots
\\
&&&&&
\\
\overline{1} & 1 & a_1 & a_2 && 1
\\
\overline{a}_1 & 1 & a_1 & a_2 & \ldots & a_1
\\
\overline{a}_2 & 1 & a_1 & a_2 &  & a_2
\\
\vdots & \vdots & \vdots & \vdots && \vdots
\\
&&&&&     
     \\};
%
\draw[thick, color=gray!80] (0.75,-0.25-2.85)--(0.75,2.6-2.85);
\draw[thick, color=gray!80] (0,-0.25-2.85)--(0,2.6-2.85);
\draw[thick, color=gray!80] (-0.75,-0.25-2.85)--(-0.75,2.6-2.85);
%
%
\draw[thick] (-1.5,-0.25)--(1.5,-0.25)--(1.5,2.6)--(-1.5,2.6)--(-1.5,-0.25);
\draw[thick] (-1.5,-0.25-2.85)--(1.5,-0.25-2.85)--(1.5,2.6-2.85)--(-1.5,2.6-2.85)--(-1.5,-0.25-2.85);
%
%
\draw[thick] (1.5,-0.25)--(1.5,2.6)--(2.2,2.6)--(2.2,-0.25)--(1.5,-0.25);
\draw[thick] (1.5,-0.25-2.85)--(1.5,2.6-2.85)--(2.2,2.6-2.85)--(2.2,-0.25-2.85)--(1.5,-0.25-2.85);
%
\draw[thick] [decorate,decoration={brace,amplitude=10pt},xshift=0pt,yshift=-4pt,thick]
(-1.4,3.3) -- (2.2,3.3) node [black,midway,yshift=-0.6cm] 
{};
\node (J) at (0.4,3.9) {$J$};
%
\draw[thick] [decorate,decoration={brace,amplitude=10pt},xshift=0pt,yshift=-4pt,thick]
(-2.2,-2.8) -- (-2.2,2.6) node [black,midway,yshift=-0.6cm] 
{};
\node (I) at (-2.9,-0.2) {$I$};
\end{tikzpicture}
\end{center}
\caption{The relationship between the generators $A$ and generators $F = \{ f_{ij} : i \in I, j \in J \}$ of $G$, where $I = A_1 \cup \overline{A}_1$ and $J = A_1 \cup \{ \infty \}$.}
\label{A-vs-fij}
\end{figure}
%
%
%
\section{Proof of the Undecidability Result}
\label{sec:undec}

The aim of this section is to prove Theorem~\ref{thm_Memb}, which will follow from Proposition~\ref{prop_bigone} below.  

As it was already mentioned in the previous section, the fact that $B_{G,H}$ has five $\D$-classes is reflected in five
corresponding regular $\D$-classes of $\ig{\B_{G,H}}$. Furthermore,
 each $\D$-class of $B_{G,H}$ is a rectangular band, so  
each $\D$-class  of  $\ig{\B_{G,H}}$ is
completely simple, and so isomorphic to some Rees matrix semigroup (without a zero element) over any of the (mutually
isomorphic) maximal subgroups of that $\D$-class. 
We have seen above that all of the $\D$-classes $\overline{L}_G$, $\overline{K}_H$ and $\overline{\{0\}}$ have trivial maximal subgroups and so they are each rectangular bands, isomorphic to $L_G$, $K_H$ and $\{ 0 \}$, respectively. We have also seen above that 
 $\overline{K}_G^{(1)}$ and $\overline{K}_G^{(2)}$ are isomorphic completely simple semigroups each with maximal subgroup isomorphic to $G$. 
We now determine the
structure matrices in the Rees matrix representations of completely simple subsemigroups $\overline{K}_G^{(1)}$ and $\overline{K}_G^{(2)}$ of $\ig{\B_{G,H}}$.

Let $F = \{ f_{ij} \ : \ i\in I,\ j\in J \}$,
where, as before, $I=A_1 \cup \overline{A}_1$, $J=A_1 \cup \{ \infty \}$.
By Lemma~\ref{rho-computable} we have computable mappings, defined by the equations \eqref{pidef} and \eqref{rhodef}, where: 
\begin{align*}
\pi^{(i)} :  (K_G^{(i)})^* \rightarrow I \times (F \cup F^{-1})^* \times J, \\
\rho^{(i)} :  I \times (F \cup F^{-1})^* \times J \rightarrow (K_G^{(i)})^*,
\end{align*}
which induce mutually inverse isomorphisms between $\overline{K}_G^{(i)}$ and the fixed Rees matrix semigroup $\mathcal{M}[G; I, J; P]$ where $P = (f_{ij}^{-1})_{j \in J, i \in I}$. So, taking the transpose of the table in Figure~\ref{A-vs-fij} and inverting all the entries gives a Rees structure matrix for the isomorphic completely simple semigroups $\overline{K}_G^{(1)}$ and $\overline{K}_G^{(2)}$. 

The main result we want to establish is the following. 

\begin{pro}\label{prop_bigone}
For all $w \in (F \cup F^{-1})^*$ the equality
\[
(1^{(1)},1^{(1)})(1^{(2)},1^{(2)}) = \rho^{(1)}(1,w^{-1},1) \rho^{(2)}(1,w,1)
\]
holds in $\ig{\B_{G,H}}$ if and only if $w$ represents an element of $H$. 
\end{pro}

\begin{rmk}
Once again, a comment about the nature of the above equality: it purportedly holds in $\ig{\B_{G,H}}$, which in turn is defined by a presentation (specifically, \eqref{eq:igE}).
Accordingly, the two terms should be interpreted as words over the generating set for $\ig{\B_{G,H}}$, which is in fact the entirety of elements of $B_{G,H}$;
however, when treating it as an alphabet, we will denoted this set by $E$.
The elements $(1^{(1)},1^{(1)})$ and  $(1^{(2)},1^{(2)})$ are single idempotents (belonging to $K_G^{(1)}$ and $K_G^{(2)}$ respectively), while $\rho^{(1)}(1,w^{-1},1) $ and $\rho^{(2)}(1,w,1)$ are words $E^\ast$ by the definition of $\rho$.
\end{rmk}

To establish this proposition, we first need a criterion for certain equalities of words from $E^*$.
Given words $u_1, u_2 \in E^*$ both representing elements of $\overline{K}_G^{(1)}$ in $\ig{\B_{G,H}}$, and words $v_1, v_2 \in E^*$ both representing elements of $\overline{K}_G^{(2)}$, we want to know when the equality $u_1 v_1 = u_2 v_2$ holds in $\ig{\B_{G,H}}$. 
To this end, for $i=1,2$, let $\mathcal{L}(\overline{K}_G^{(i)})$  denote the set of words over $E$ representing elements of $\overline{K}_G^{(i)}$ in $\ig{\B_{G,H}}$. The following lemma describes these two sets. 
\begin{lem}\label{lem_reps}
Let $w \in E^*$ and $i\in\{1,2\}$. Then 
$w$ represents an element of $\overline{K}_G^{(i)}$ if and only if $w \in (L_G \cup K_H \cup K_G^{(i)})^*$ and $w$ contains at least one letter from $K_G^{(i)}$.
\end{lem}

\begin{proof}
We prove the assertion for $i=1$; the case $i=2$ is analogous.
According to Theorem~\ref{thm:testing_reg}, $w$ represents an element of 
$\overline{K}_G^{(1)}$ if and only if it has the form $w=uev$, where 
$e\in K_G^{(1)}$, and $ue\L e$ and $ev\R e$ in $B_{G,H}$.
From the definition of $B_{G,H}$ (and Figure \ref{fig_BGH}), it is clear that left/right multiplying an idempotent $e^\prime\in K_G^{(1)}$
by any element from $L_G\cup K_H\cup K_G^{(1)}$ results in an idempotent in the $\L$-/$\R$-class of $e^\prime$, whereas multiplication by an element from  $K_G^{(2)}\cup\{0\}$ yields $0$.
It follows that precisely the words from $(L_G\cup K_H\cup K_G^{(1)})^\ast$
are available for $u$ and $v$, yielding $w \in (L_G \cup K_H \cup K_G^{(1)})^*$,
with at least one letter (namely $e$) from $K_G^{(1)}$, as required.
\end{proof}

\begin{lem}\label{lem_remarkable} 
Let
$
(u,v) \in \mathcal{L}(\overline{K}_G^{(1)}) \times \mathcal{L}(\overline{K}_G^{(2)}).
$
If $u \equiv u'f$ where $f \in L_G \cup K_H$ then $(u',fv) \in \mathcal{L}(\overline{K}_G^{(1)}) \times \mathcal{L}(\overline{K}_G^{(2)})$. Similarly, if $v \equiv fv'$ where $f \in L_G \cup K_H$ then 
$(uf,v') \in \mathcal{L}(\overline{K}_G^{(1)}) \times \mathcal{L}(\overline{K}_G^{(2)})$. 
\end{lem}
\begin{proof}
This is an immediate consequence of Lemma~\ref{lem_reps}
\end{proof}
This lemma forms the basis of some relations which we shall now define on the set $\mathcal{L}(\overline{K}_G^{(1)}) \times \mathcal{L}(\overline{K}_G^{(2)})$.

\begin{dfn}\label{def_approx}
For  $(u_1,v_1), (u_2,v_2) \in \mathcal{L}(\overline{K}_G^{(1)}) \times \mathcal{L}(\overline{K}_G^{(2)})$ we write 
\[
(u_1,v_1) \approx (u_2,v_2)
\]
if and only if  one of the following three conditions is satisfied:
\begin{enumerate}
\item[(i)]
$u_1=u_2$ and $v_1=v_2$ hold in $\ig{\B_{G,H}}$;
\item[(ii)] $u_1 = u_2 f$ and
$fv_1 = v_2$ hold in $\ig{\B_{G,H}}$ for some $f \in L_G \cup K_H$;
\item[(iii)] $u_1f=u_2$ and 
 $v_1 =fv_2$ hold in $\ig{\B_{G,H}}$ for some $f \in L_G \cup K_H$. 
\end{enumerate}
\end{dfn}
Note that $\approx$ is a reflexive and 
symmetric relation on the set $\mathcal{L}(\overline{K}_G^{(1)}) \times \mathcal{L}(\overline{K}_G^{(2)})$. 
Now let $\sim$ denote the  transitive closure of the relation $\approx$ on $\mathcal{L}(\overline{K}_G^{(1)}) \times \mathcal{L}(\overline{K}_G^{(2)})$. Note that $(u_1, v_1) \sim (u_2, v_2)$ does not imply $u_1 = u_2$ or $v_1 = v_2$ in $\ig{\B_{G,H}}$. However, if $u_1 = u_2$ and $v_1 = v_2$ in $\ig{\B_{G,H}}$ then by definition $(u_1,v_1) \sim (u_2,v_2)$. Thus the equivalence relation $\sim$ on $\mathcal{L}(\overline{K}_G^{(1)}) \times \mathcal{L}(\overline{K}_G^{(2)})$ induces a well-defined equivalence relation on the set $\overline{K}_G^{(1)} \times \overline{K}_G^{(2)}$. 

\begin{lem}\label{lem_claimA}
Let $(\alpha, \beta), (\alpha', \beta') \in \mathcal{L}(\overline{K}_G^{(1)}) \times \mathcal{L}(\overline{K}_G^{(2)})$. If $\alpha \beta \equiv \alpha' \beta'$ then 
$(\alpha, \beta) \sim (\alpha', \beta')$. 
\end{lem}
\begin{proof}
Without loss of generality we may suppose that $|\alpha'| > |\alpha|$. Write $\alpha' \equiv \alpha\gamma$ and $\beta \equiv \gamma\beta'$. Since 
\begin{align*}
& \alpha\gamma \equiv \alpha' \in \mathcal{L}(\overline{K}_G^{(1)}) \subseteq (L_G \cup K_H \cup K_G^{(1)})^* \ \mathrm{and}  \\
& \gamma\beta' \equiv \beta \in \mathcal{L}(\overline{K}_G^{(2)}) \subseteq (L_G \cup K_H \cup K_G^{(2)})^*, 
\end{align*}
it follows that $\gamma \in (L_G \cup K_H)^*$. Write $\gamma \equiv \gamma_1 \gamma_2 \ldots \gamma_r$ where each $\gamma_j \in L_G \cup K_H$. Then 
\[
(\alpha, \beta) =
(\alpha, \gamma_1 \gamma_2 \ldots \gamma_r \beta') \approx
(\alpha \gamma_1,  \gamma_2 \ldots \gamma_r \beta') \approx \ldots \approx
(\alpha \gamma_1 \gamma_2 \ldots \gamma_r,   \beta') = (\alpha', \beta'), 
\]
where each pair in this sequence belongs to $ \mathcal{L}(\overline{K}_G^{(1)}) \times \mathcal{L}(\overline{K}_G^{(2)})$ by Lemma~\ref{lem_remarkable}.
\end{proof}

\begin{lem}\label{lem_claimB}
Let $(u,v) \in  \mathcal{L}(\overline{K}_G^{(1)}) \times \mathcal{L}(\overline{K}_G^{(2)})$. Let $w \in E^*$ be a word obtained from the word $uv$ by the application of a single relation from the presentation for $\ig{\B_{G,H}}$. Then the word $w$ admits a decomposition $w \equiv \alpha\beta$ where $(\alpha, \beta) \in \mathcal{L}(\overline{K}_G^{(1)}) \times \mathcal{L}(\overline{K}_G^{(2)})$ and $(u,v) \sim (\alpha, \beta)$. 
\end{lem}
\begin{proof}
Recall that the defining relations of $\ig{\B_{G,H}}$ are all of the form $ef = g$ where $(e, f)$ is a basic pair (meaning $\{e, f \} \cap \{ef, fe\} \neq \varnothing$) and $ef=g$ holds in the band $B_{G,H}$. 
Keeping in mind that $\mathcal{L}(\overline{K}_G^{(1)})$ is the set of all words from  $E^*$ representing elements of $\overline{K}_G^{(1)}$, 
if the relation that is applied to transform $uv$ into $w$ is applied entirely within the subword $u$, transforming it into $u^\prime$, and hence
transforming $uv$ into $u'v \equiv w$ where $u'=u$ in $\ig{\B_{G,H}}$, then the lemma trivially holds by taking $\alpha \equiv u'$ and $\beta \equiv v$. Similarly the lemma is easily seen to hold if  the relation is applied entirely within the subword $v$. 
The only remaining case left to consider is when $u \equiv u'e$, $v \equiv fv'$, $w \equiv u'gv'$ and $w$ is obtained from $uv$ by applying the relation $ef=g$ from the presentation of $\ig{\B_{G,H}}$. Since 
\begin{align*}
& u'e \equiv u \in \mathcal{L}(\overline{K}_G^{(1)}) \subseteq (L_G \cup K_H \cup K_G^{(1)})^* \ \mathrm{and}  \\
& fv' \equiv v \in \mathcal{L}(\overline{K}_G^{(2)}) \subseteq (L_G \cup K_H \cup K_G^{(2)})^*, 
\end{align*}
it follows that $e \in (L_G \cup K_H) \cup K_G^{(1)}$ and $f \in (L_G \cup K_H) \cup K_G^{(2)}$. As $\{e, f \}$ is a basic pair, and the $\D$-classes $K_G^{(1)}$ and $K_G^{(2)}$ are incomparable in the $\J$-class order in the band $B_{G,H}$, it follows that we  cannot have both $e \in K_G^{(1)}$ and $f \in K_G^{(2)}$. This means that we must have $e \in L_G \cup K_H$ or $f \in L_G \cup K_H$ (or possibly both). 

Suppose first that $e \in L_G \cup K_H$. Then 
setting $\alpha \equiv u'$ and $\beta \equiv gv'$
we have $w \equiv u'gv' \equiv \alpha\beta$ and 
\[
(u,v) = (u'e, fv') \approx (u', efv') \approx (u', gv') = (\alpha, \beta). 
\]
Here $(u', efv') \in \mathcal{L}(\overline{K}_G^{(1)}) \times \mathcal{L}(\overline{K}_G^{(2)})$ by Lemma~\ref{lem_remarkable} since $e \in L_G \cup K_H$.

Finally, the case that $f \in L_G \cup K_H$ follows by a dual argument. 
This deals with all possible cases, and thus completes the proof. 
\end{proof}

The following lemma shows the importance of the relation $\sim$ in connection with the word problem. 

\begin{lem}\label{lem_SimRelation}
Let 
$
(u_1, v_1), (u_2, v_2) \in \mathcal{L}(\overline{K}_G^{(1)}) \times \mathcal{L}(\overline{K}_G^{(2)}).
$
Then $u_1v_1 = u_2v_2$ in 
$\ig{\B_{G,H}}$ if and only if $(u_1, v_1) \sim (u_2, v_2)$. 
\end{lem}
\begin{proof}
($\Leftarrow$) Clearly it suffices to prove if $(u_1, v_1) \approx (u_2, v_2)$ then $u_1 v_1 = u_2 v_2$, so suppose that $(u_1, v_1) \approx (u_2, v_2)$. In all three cases (i)--(ii) of Definition~\ref{def_approx} it is immediately seen that $u_1 v_1 = u_2 v_2$ in $\ig{\B_{G,H}}$.

($\Rightarrow$) Let 
$
(u_1, v_1), (u_2, v_2) \in \mathcal{L}(\overline{K}_G^{(1)}) \times \mathcal{L}(\overline{K}_G^{(2)}),
$
and suppose that $u_1v_1 = u_2v_2$ in $\ig{\B_{G,H}}$. This means there is a sequence of words from $E^*$ 
\[
u_1v_1 \equiv w_1 = w_2 = \ldots = w_k \equiv u_2v_2,
\]
such that for each $j$, the word $w_{j+1}$ is obtained from $w_j$ by the application of a single relation from the defining relations of $\ig{\B_{G,H}}$. Working along this sequence and repeatedly applying Lemma~\ref{lem_claimB} we conclude that for each $1 \leq j \leq k$ the word $w_j$ admits a decomposition $w_j \equiv \alpha_j \beta_j$ such that $(\alpha_j, \beta_j) \in \mathcal{L}(\overline{K}_G^{(1)}) \times \mathcal{L}(\overline{K}_G^{(2)})$ and 
\[
(u_1, v_1) \sim (\alpha_1, \beta_1) \sim (\alpha_2, \beta_2) \sim 
\ldots \sim (\alpha_k, \beta_k). 
\]
Finally, $u_2v_2 \equiv w_k \equiv \alpha_k\beta_k$ 
with $(u_2, v_2), (\alpha_k, \beta_k) \in \mathcal{L}(\overline{K}_G^{(1)}) \times \mathcal{L}(\overline{K}_G^{(2)})$ 
and so by Lemma~\ref{lem_claimA} it follows that $(\alpha_k, \beta_k) \sim (u_2, v_2)$, and thus $(u_1, v_1) \sim (u_2, v_2)$.  
\end{proof}

Next we record how the elements from 
$L_G \cup K_H$ act on $\overline{K}_G^{(1)}$ and
$\overline{K}_G^{(2)}$ in $\ig{\B_{G,H}}$. 

\begin{lem}
\label{lemeactions}
Let $e = (\sigma, \tau) \in K_H \cup L_G$ and let 
$(i, w, j) \in 
I \times (F \cup F^{-1})^* \times J$. Further choose arbitrary $i_0 \in \im(\sigma)$ and $j_0 \in \im(\tau)$. Then
\[
\rho^{(1)}(i,w,j)\;e =
\rho^{(1)}(i, \ 
w f_{i_0 j}^{-1} f_{i_0, j\tau}, 
\ 
j\tau)
\]
and
\[
e\;\rho^{(2)}(i,w,j) =
\rho^{(2)}
(\sigma(i), \
f_{\sigma(i), j_0} f_{ij_0}^{-1} w, \
j
)
\]
both hold in $\ig{\B_{G,H}}$. 
\end{lem}
\begin{proof}
We prove the first equality only; the second follows by a dual argument. 
Since $i_0 \in \im \sigma$ the rules of multiplication given by \eqref{eqn_multiplication} imply that
\begin{align}
e(i_0^{(1)},j^{(1)})=(i_0^{(1)},j^{(1)}) \  \mbox{and}   \ (i_0^{(1)},j^{(1)})e=(i_0^{(1)},j\tau^{(1)}) \label{dagger}
\end{align}
in $B_{G,H}$. 
Hence $\{e, (i_0^{(1)},j^{(1)}) \}$ is a basic pair, and the above relations are among the defining relations for $\ig{\B_{G,H}}$. 
By the definition of the mapping $\rho$ in equation \eqref{rhodef} we have 
\begin{equation}
\label{almostthere}
\rho^{(1)}(i,w,j)\equiv (i_1^{(1)},j_1^{(1)})\dots (i_{k-1}^{(1)},j_{k-1}^{(1)}) (i_k^{(1)},j_k^{(1)})\equiv  u (i_k^{(1)},j_k^{(1)})
\end{equation}
for some $i_1, \ldots, i_k \in I$, $j_1, \ldots, j_k \in J$. Notice that necessarily $i=i_1$ and $j=j_k$. 
Also, recall that the structure matrix for the Rees matrix representation of $\overline{K}_G^{(1)}$ is given by $P=(f_{ij}^{-1})_{j\in J,i\in I}$.
It follows that the element $(i_0,f_{i_0,j\tau},j\tau)$ is an idempotent in this Rees matrix semigroup, and hence
\begin{equation}
\label{almostthere1}
(i_0^{(1)},j\tau^{(1)})=\rho^{(1)}(i_0,f_{i_0,j\tau},j\tau)
\end{equation}
in $\ig{\B_{G,H}}$.
Now we have
\begin{align*}
\rho^{(1)}(i,w,j) e & \equiv  
u (i_k^{(1)},j_k^{(1)}) e 
\equiv 
u (i_k^{(1)}, j^{(1)}) e  && \mbox{(by \eqref{almostthere})}\\
& = 
u ((i_k^{(1)}, j^{(1)}) (i_0^{(1)}, j^{(1)})) e & & \mbox{(since $(i_k^{(1)}, j^{(1)}) \L (i_0^{(1)}, j^{(1)})$)} \\
& = 
u (i_k^{(1)}, j^{(1)}) ((i_0^{(1)}, j^{(1)}) e) 
\\
& = 
u (i_k^{(1)}, j^{(1)}) (i_0^{(1)},j\tau^{(1)}) & & \mbox{(by \eqref{dagger})}
\\
& = 
\rho^{(1)}(i,w,j) (i_0^{(1)},j\tau^{(1)}) \\
& = 
\rho^{(1)}(i,w,j) 
\rho^{(1)}(i_0, f_{i_0, j\tau}, j\tau)
 & & \mbox{(by \eqref{almostthere1})} \\
& =   \rho^{(1)}(i, \ 
w f_{i_0 j}^{-1} f_{i_0, j\tau}, 
\ 
j\tau) & & \mbox{(by \eqref{eq_rhoformula}).}
\end{align*}
as required.
\end{proof}

\begin{lem}\label{lem_910a}
Suppose $u_1, u_2 \in \mathcal{L}(\overline{K}_G^{(1)})$, $v_1, v_2 \in \mathcal{L}(\overline{K}_G^{(2)})$ are such that $(u_1, v_1) \approx (u_2, v_2)$. If 
$v_1 = \rho^{(2)}(i_1, w_1, j_1)$ and $v_2 =  \rho^{(2)}(i_2, w_2, j_2)$ in $\ig{\E}$  
then the word $w_1 w_2^{-1} \in (F \cup F^{-1})^*$ represents an element of $H$. 
\end{lem}
\begin{proof}
Without loss of generality we assume that $u_1 \equiv u_2e$, $v_2 \equiv ev_1$, where $e = (\sigma, \tau) \in L_G \cup K_H$. By Lemma~\ref{lemeactions} we have
\[
 \rho^{(2)}(i_2, w_2, j_2)  =  
 v_2 = e v_1 = e \rho^{(2)}(i_1, w_1, j_1) 
 =  \rho^{(2)}(\sigma(i_1), f_{\sigma(i_1), j_0} f_{i_1 j_0}^{-1} w_1, j_1),
\]
where $j_0 \in \im \tau$, 
and it follows that $w_1 w_2^{-1} = f_{i_1 j_0} f_{\sigma(i_1), j_0}^{-1}$. 
Now, if $e \in L_G$, we must have $j_0 \in A_1$ and both $i_1$ and $\sigma(i_1)$ belong to one of $A_1$ or $\overline{A}_1$ (see Table~\ref{table_table}). In either case we have $f_{i_1 j_0} = f_{\sigma(i_1), j_0}$ (see Figure~\ref{A-vs-fij}) and hence $w_1w_2^{-1}$ is equal to the empty word, and so represents the identity element, which belongs to $H$. 
If on the other hand, $e = (a,b) \in K_H$ then we must have $j_0 = b \in B_1$. 
From Figure~\ref{A-vs-fij} we see that the only entries in the column $b$ are $1$ and $b$. Hence 
$f_{i_1 j_0}, f_{\sigma(i_1), j_0} \in \{1, b \}$, implying  $w_1 w_2^{-1} \in \{ 1, b, b^{-1} \} \subseteq B_1$. Since every element of $\{ 1, b, b^{-1} \}$ represents an element of $H$, this completes the proof. 
\end{proof}

\begin{lem}\label{lem_910b}
Suppose $u_1, u_2 \in \mathcal{L}(\overline{K}_G^{(1)})$, $v_1, v_2 \in \mathcal{L}(\overline{K}_G^{(2)})$ are such that $(u_1, v_1) \sim (u_2, v_2)$. If 
$v_1 = \rho^{(2)}(i_1, w_1, j_1)$ and $v_2 =  \rho^{(2)}(i_2, w_2, j_2)$ in $\ig{\E}$  
then $w_1 w_2^{-1} \in (F \cup F^{-1})^*$ represents an element of $H$. 
\end{lem}
\begin{proof}
By definition of $\sim$, there exists a sequence
\[
(u_1,v_1)=(x_1,y_1)\approx (x_2,y_2)\approx\dots\approx (x_t,y_t)=(u_2,v_2).
\]
For $1 \leq r \leq t$ let $k_r \in I$, $l_r \in J$ and $z_r \in (F \cup F^{-1})^*$ be such that 
$y_r = \rho^{(2)}(k_r,z_r,l_r)$. 
Then by Lemma \ref{lem_910a} the word $z_r z_{r+1}^{-1}$ represents an element of $H$ for all $r=1,\dots,t$.
Hence, the product 
\[
(z_1z_2^{-1})(z_2z_3^{-1})\dots (z_{t-1}z_t^{-1})
\]
also represents an element of $H$. But 
\[
z_1z_2^{-1}z_2z_3^{-1}\dots z_{t-1}z_t^{-1} = z_1z_t^{-1} = w_1w_2^{-1}
\]
in $G$ and we conclude that $w_1w_2^{-1}$ represents an element of $H$. 
%
\end{proof}

\begin{lem}\label{lem:H} 
If the equality
\[
\rho^{(1)}(i, w_1, j_1) \rho^{(2)}(i_1, w_2, j) =
\rho^{(1)}(i, w_3, j_2) \rho^{(2)}(i_2, w_4, j)
\]
holds in $\ig{\B_{G,H}}$, for some $i, i_1, i_2 \in I$, $j, j_1, j_2 \in J$ and 
$w_1, w_2, w_3, w_4 \in (F \cup F^{-1})^*$, then $w_2 w_4^{-1}$ represents an element of $H$. 
\end{lem}

\begin{proof}
Set $u_1\equiv\rho^{(1)}(i, w_1, j_1)$,
$v_1\equiv\rho^{(2)}(i_1, w_2, j)$, $u_2\equiv\rho^{(1)}(i, w_3, j_2)$, $v_2\equiv\rho^{(2)}(i_2, w_4, j)$. Then we have
$
(u_1, v_1), (u_2, v_2) \in \mathcal{L}(\overline{K}_G^{(1)}) \times \mathcal{L}(\overline{K}_G^{(2)})
$
and since $u_1v_1 = u_2v_2$ in $\ig{\B_{G,H}}$ it follows from Lemma~\ref{lem_SimRelation} that $(u_1, v_1) \sim (u_2, v_2)$. Then applying 
Lemma~\ref{lem_910b} we conclude that $w_2 w_4^{-1}$ represents an element of $H$. 
\end{proof}

\begin{lem}
\label{b-1b}
For any two words $w_1,w_2\in (F\cup F^{-1})^\ast$ and any $b\in B$, the relation
\[
\rho^{(1)} (1,w_1,1) \rho^{(2)} (1,w_2,1)=
\rho^{(1)} (1,w_1b^{-1},1) \rho^{(2)} (1,bw_2,1)
\]
holds in $\ig{\B_{G,H}}$.
\end{lem}

\begin{proof}
According to Lemma \ref{lem_SimRelation}, it suffices to prove that
\[
(\rho^{(1)} (1,w_1,1), \rho^{(2)} (1,w_2,1)) \sim
(\rho^{(1)} (1,w_1b^{-1},1) ,\rho^{(2)} (1,bw_2,1)).
\]
We do this by demonstrating
\begin{align*}
(\rho^{(1)} (1,w_1,1), \rho^{(2)} (1,w_2,1))
&\approx 
(\rho^{(1)} (1,w_1,b), \rho^{(2)} (\overline{1},bw_2,1))
\\
&\approx 
(\rho^{(1)} (1,w_1,\infty), \rho^{(2)} (\overline{b},bw_2,1))
\\
&\approx 
(\rho^{(1)} (1,w_1b^{-1},1), \rho^{(2)}(\overline{b},bw_2,1))
\\
&\approx 
(\rho^{(1)} (1,w_1b^{-1},1), \rho^{(2)} (1,bw_2,1)).
\end{align*}
The four `intervening' idempotents $(\sigma_1,\tau_1),(\sigma_2,\tau_2),(\sigma_3,\tau_3),(\sigma_4,\tau_4)$
used to establish these $\approx$-relationships are
$(1,b)\in K_H$, $e_b\in L_G$, $e_{\overline{b}}\in L_G$, and $(1,1)\in K_H$ respectively.
The actual computations, which all rely on Lemma \ref{lemeactions} and Figure~\ref{A-vs-fij}, are as follows:
\begin{align*}
&
\rho^{(1)}(1,w_1,1)(\sigma_1,\tau_1)=\rho^{(1)}(1,w_1 f_{11}^{-1}f_{1b},b)=\rho^{(1)}(1,w_1,b)
\\
&
(\sigma_1,\tau_1)\rho^{(2)}(\overline{1},bw_2,1)=
\rho^{(2)}(1,f_{1b}f_{\overline{1},b}^{-1}bw_2,1)
=\rho^{(2)}(1,1\cdot b^{-1}\cdot bw_2,1)=\rho^{(2)}(1,w_2,1)
\\
&
\rho^{(1)}(1,w_1,\infty)(\sigma_2,\tau_2)=\rho^{(1)}(1,w_1 f_{1,\infty}^{-1}f_{1b},b)=\rho^{(1)}(w_1\cdot 1\cdot 1,b)=\rho^{(1)}(1,w_1,b)
\\
&
(\sigma_2,\tau_2)\rho^{(2)} (\overline{1},bw_2,1)=\rho^{(2)}
(\overline{b},f_{\overline{b},1}f_{\overline{1},1}^{-1}bw_2,1)
=\rho^{(2)}(\overline{b},1\cdot 1\cdot bw_2,1)=\rho^{(2)}(\overline{b},bw_2,1)
\\
&
\rho^{(1)}(1,w_1,\infty)(\sigma_3,\tau_3)=\rho^{(1)}(1,w_1 f_{b,\infty}^{-1}f_{bb},1)=\rho^{(1)}(w_1\cdot b^{-1}\cdot 1,1)=\rho^{(1)}(1,w_1b^{-1},1)
\\
&
(\sigma_3,\tau_3)\rho^{(2)} (\overline{b},bw_2,1)=\rho^{(2)}
(\overline{b},f_{\overline{b},1}f_{\overline{b},1}^{-1}bw_2,1)
=\rho^{(2)}(\overline{b},bw_2,1)
\\
&
\rho^{(1)}(1,w_1b^{-1},1)(\sigma_4,\tau_4)=\rho^{(1)}(1,w_1 b^{-1} f_{11}^{-1}f_{11},1)=
\rho^{(1)}(1,w_1b^{-1},1)
\\
&
(\sigma_4,\tau_4)\rho^{(2)}
(\overline{b},bw_2,1)=\rho^{(2)} (1,f_{11}f_{\overline{b},1}^{-1}bw_2,1)
=\rho^{(2)}(1,1\cdot 1\cdot bw_2,1)=\rho^{(2)}(1,bw_2,1).
\end{align*}
This proves the lemma.
\end{proof}

\begin{proof}[Proof of Proposition \ref{prop_bigone}]
($\Rightarrow$)
The given equality can be written as
\[
\rho^{(1)} (1,1,1)\rho^{(2)}(1,1,1)=
\rho^{(1)} (1,w^{-1},1)\rho^{(2)} (1,w,1),
\]
and it follows from Lemma \ref{lem:H} that $1w^{-1}$, and hence $w$ itself, both represent elements of $H$.
\smallskip

($\Leftarrow$)
Suppose $w$ represents an element of $H$, and write $w=b_1\dots b_k$, a product of generators from $B$.
Repeatedly applying Lemma \ref{b-1b} we have
\begin{align*}
\rho^{(1)}(1,1,1)\rho^{(2)}(1,1,1)
&=
\rho^{(1)} (1,b_k^{-1},1)\rho^{(2)}(1,b_k,1)\\
&=\rho^{(1)} (1,b_k^{-1}b_{k-1}^{-1},1)\rho^{(2)}(1,b_{k-1}b_k,1)
\\
&= \dots  \; \; = \rho^{(1)} (1,b_k^{-1}\dots b_{1}^{-1},1)\rho^{(2)}(1,b_{1}\dots b_k,1)\\
&=\rho^{(1)} (1,w^{-1},1)\rho^{(2)}(1,w,1),
\end{align*}
a sequence of equalities valid in $\ig{\B_{G,H}}$.
\end{proof}

We are now finally in the position to prove Theorem \ref{thm_Memb}, which asserts that if $\ig{\B_{G,H}}$ has decidable word problem then the membership problem for $H$ in $G$ is decidable.

\begin{proof}[Proof of Theorem \ref{thm_Memb}]
Suppose that $\ig{\B_{G,H}}$ has decidable word problem.
Recall that $F$ is a finite generating set for $G$, and 
consider an arbitrary $w\in (F\cup F^{-1})^\ast$.
By Lemma \ref{rho-computable} the words
$\rho^{(1)}(1,w^{-1},1),\rho^{(2)}(1,w,1)\in E^\ast$ are
effectively computable.
By Proposition  \ref{prop_bigone} the word $w$ represents an element of $H$ if and only if
\[
(1^{(1)},1^{(1)}) ( 1^{(2)},1^{(2)}) = \rho^{(1)}(1,w^{-1},1)\rho^{(2)}(1,w,1)
\]
holds in $\ig{\B_{G,H}}$. This equality in turn can be checked using the decision algorithm for the word problem
for $\ig{\B_{G,H}}$.
\end{proof}

\section{Some remarks on Sch\"{u}tzenberger groups and further questions}\label{sec_conc}

The authors hope that the present article will mark the transition of focus in research on free idempotent 
generated semigroups from the regular to the non-regular part of $\ig{\E}$. One may anticipate that the next 
stage is to analyse the so-called Sch\"{u}tzenberger groups of $\H$-classes in non-regular $\D$-classes. These 
were originally introduced by Sch\"{u}tzenberger \cite{Sch1,Sch2} as `virtual' counterparts to maximal subgroups 
in regular $\D$-classes. In fact, our construction described in Sections \ref{sec_undec1}--\ref{sec:undec}
already gives some information relevant to such analysis which seems worth recording.

Let $D$ be an arbitrary $\D$-class of  $S$, and let $X\subseteq D$ be one of its $\H$-classes. Then the set of all elements 
$s\in S^1$ such that $Xs\subseteq X$ is denoted by  $\mathrm{Stab}(X)$ and called the  the (\emph{right}) \emph{stabiliser} 
of $X$. Then Green's Lemma \cite[Lemma 2.2.1]{HoBook} ensures that $s\in\mathrm{Stab}(X)$ actually implies $Xs=X$ and that 
the right translation $\rho_s:x\mapsto xs$, $x\in X$, is a permutation of $X$. Define an equivalence $\sigma_X$ on 
$\mathrm{Stab}(X)$ by $(s,t)\in\sigma_X$ is and only if $\rho_s=\rho_t$ (i.e.\ $s,t$ induce  the same permutation 
on $X$). Then the quotient set $\mathrm{Stab}(X)/\sigma_X$ is naturally identified with the collection of permutations 
$\{\rho_s: s\in\mathrm{Stab}(X)\}$, and the latter is easily seen to be a group: this is the \emph{(right) 
Sch\"{u}tzenberger group} $\Gamma_X$ of $X$. Here are several basic facts about Sch\"{u}tzenberger groups:
\begin{itemize}
\item $\Gamma_X$ acts regularly on $X$; consequently $|\Gamma_X|=|X|$;
\item If $X$ is a group then $\Gamma_X\cong X$;
\item If $X'$ is an $\H$-class contained in the same $\D$-class as $X$ then $\Gamma_X\cong\Gamma_{X'}$.
\end{itemize}
We refer to \cite[Section 2.3]{LaBook} for proofs of these facts. 

We believe that the following result sheds a bit more light on the `background' of the undecidability result 
proved in this paper.

\begin{pro}\label{pro:Schutz-H}
If $X$ is the $\H$-class of the element $$\rho^{(1)}(1,1,1)\rho^{(2)}(1,1,1)$$ in $\ig{\B_{G,H}}$, then the 
Sch\"{u}tzenberger group $\Gamma_X$ of $X$ is isomorphic to $H$.
\end{pro}

It is worth noticing that this immediately implies a `non-regular analogue' of the main result of \cite{DR}.

\begin{thm} 
Any (finitely presented) group arises as a Sch\"{u}tzenberger group of an $\H$-class belonging to a non-regular $\D$-class 
in a free idempotent generated semigroup over a (finite) band. 
\end{thm}

Now we verify Proposition \ref{pro:Schutz-H}.

\begin{lem}\label{lem-uv}
Assume 
$$ 
\rho^{(1)}(1,1,1)\rho^{(2)}(1,1,1) = \rho^{(1)}(1,u,j)\rho^{(2)}(i,v,1) 
$$ 
holds in $\ig{\B_{G,H}}$ for some $u,v\in (F\cup F^{-1})^\ast$, $i\in I$, $j\in J$. Then $vu=f_{ij}$ holds in $G$. 
\end{lem}

\begin{proof} 
It suffices to prove the following 
\begin{cla} 
Assume that we have $\rho^{(1)}(1,u,j)\rho^{(2)}(i,v,1)\approx \rho^{(1)}(1,u',j')\rho^{(2)}(i',v',1)$ via an idempotent 
$e=(\sigma,\tau)\in L_G\cup K_H$, as described in Definition \ref{def_approx}. 
If $vu=f_{ij}$ holds in $G$, then $v'u'=f_{i'j'}$. 
\end{cla}

So, assume that
\begin{align*} 
\rho^{(1)}(1,u,j)\rho^{(2)}(i,v,1) &= [\rho^{(1)}(1,u',j')e]\rho^{(2)}(i,v,1)\\
                          &= \rho^{(1)}(1,u',j')[e\rho^{(2)}(i,v,1)]\\ 
													&= \rho^{(1)}(1,u',j')\rho^{(2)}(i',v',1). 
\end{align*}
Then $j'\tau=j$ and $\sigma (i)=i'$, so that $j\in\im(\tau)$ and $i'\in\im(\sigma)$. By Lemma \ref{lemeactions}, we have
that $u=u'f_{i'j'}^{-1}f_{i'j}$ and $v'=f_{i'j}f_{ij}^{-1}v$ holds in $G$. Hence, we have 
$$ v'u' = f_{i'j}f_{ij}^{-1}vuf_{i'j}^{-1}f_{i'j'} = f_{i'j}f_{ij}^{-1}f_{ij}f_{i'j}^{-1}f_{i'j'} = f_{i'j'}, $$ 
as required. 
\end{proof}

Lemma \ref{lem-uv} taken together with Proposition \ref{prop_bigone} gives the following strengthening of the latter.

\begin{cor} \label{cor_bigone}
The relation
$$ 
\rho^{(1)}(1,1,1)\rho^{(2)}(1,1,1) = \rho^{(1)}(1,u,1)\rho^{(2)}(1,v,1) 
$$ 
holds in $\ig{\B_{G,H}}$ for $u,v\in (F\cup F^{-1})^\ast$ if and only if 
both $u,v$ represent elements of $H$ and $u=v^{-1}$ holds in $G$. 
\end{cor}

\begin{lem}\label{lem:eq-uv} 
Let $u\in (F\cup F^{-1})^\ast$. There exists a word $v\in (F\cup F^{-1})^\ast$ such that 
$$ \rho^{(1)}(1,1,1)\rho^{(2)}(1,u,1) = \rho^{(1)}(1,v,1)\rho^{(2)}(1,1,1) $$ 
holds in $\ig{\B_{G,H}}$ if and only if $u$ represents an element of $H$, in which case $u=v$ must hold in $G$. 
\end{lem}

\begin{proof} 
If $ \rho^{(1)}(1,1,1)\rho^{(2)}(1,u,1) = \rho^{(1)}(1,v,1)\rho^{(2)}(1,1,1) $ then 
\begin{align*}
\rho^{(1)}(1,1,1)\rho^{(2)}(1,1,1) &= \rho^{(1)}(1,1,1)\rho^{(2)}(1,u,1)\rho^{(2)}(1,u^{-1},1)\\ 
                          &=\rho^{(1)}(1,v,1)\rho^{(2)}(1,1,1)\rho^{(2)}(1,u^{-1},1)\\
													&= \rho^{(1)}(1,v,1)\rho^{(2)}(1,u^{-1},1). 
\end{align*}
Thus, by the previous corollary, $v=(u^{-1})^{-1}=u$ holds in $G$ and both $u,v$ represent an element of $H$.
The converse is a direct consequence of Lemma \ref{b-1b}.
\end{proof}

\begin{lem}\label{lem-w1w2}
The elements of $X$ are precisely $\rho^{(1)}(1,1,1)\rho^{(2)}(1,w,1)$,  where $w\in (F\cup F^{-1})^\ast$ represents an element of $H$. 
Furthermore, $\rho^{(1)}(1,1,1)\rho^{(2)}(1,w_1,1)=\rho^{(1)}(1,1,1)\rho^{(2)}(1,w_2,1)$ holds in $\ig{\B_{G,H}}$ if and only if $w_1=w_2$ holds in $H$.
\end{lem}

\begin{proof} 
Any element of the $\R$-class of $\rho^{(1)}(1,1,1)\rho^{(2)}(1,1,1)$ must be of the form 
$$\rho^{(1)}(1,1,1)\rho^{(2)}(1,1,1)u= \rho^{(1)}(1,1,1)\rho^{(2)}(1,w,j)$$ 
for some words $u\in E^\ast$, $w\in (F\cup F^{-1})^\ast$, and $j\in J$. Similarly, any element of the $\L$-class of 
$ \rho^{(1)}(1,1,1)\rho^{(2)}(1,1,1)$ must be of the form $\rho^{(1)}(i,w',1)\rho^{(2)}(1,1,1)$ for some $i\in I$ and $w'\in (F\cup F^{-1})^\ast$. 
So, for such elements to be in the $\H$-class of  $\rho^{(1)}(1,1,1)\rho^{(2)}(1,1,1)$ we must have $i=j=1$ and
$$\rho^{(1)}(1,1,1)\rho^{(2)}(1,w,1) = \rho^{(1)}(1,w',1)\rho^{(2)}(1,1,1)$$
must hold in $\ig{\B_{G,H}}$. Therefore, by Lemma \ref{lem:eq-uv} we must have that $w=w'$ holds in $G$ and both these 
words represent an element of $H$. On the other hand, it is readily verified that for any $w\in (F\cup F^{-1})^\ast$ representing an 
element of $H$, the element
$$\rho^{(1)}(1,1,1)\rho^{(2)}(1,w,1) = \rho^{(1)}(1,w,1)\rho^{(2)}(1,1,1),$$ 
is $\H$-related to $\rho^{(1)}(1,1,1)\rho^{(2)}(1,1,1)$.  

Finally, if $\rho^{(1)}(1,1,1)\rho^{(2)}(1,w_1,1)= \rho^{(1)}(1,1,1)\rho^{(2)}(1,w_2,1)$ holds in  $\ig{\B_{G,H}}$ then
\begin{align*}
\rho^{(1)}(1,1,1)\rho^{(2)}(1,1,1) &= \rho^{(1)}(1,1,1)\rho^{(2)}(1,w_1,1)\rho^{(2)}(1,w_1^{-1},1)\\
                          &= \rho^{(1)}(1,1,1)\rho^{(2)}(1,w_2,1)\rho^{(2)}(1,w_1^{-1},1)\\
													&= \rho^{(1)}(1,1,1)\rho^{(2)}(1,w_2w_1^{-1},1),
\end{align*}
implying that $w_1=w_2$ holds in $H$. The converse is immediate. 
\end{proof}

\begin{proof}[Proof of Proposition \ref{pro:Schutz-H}]
Each word belonging to the stabiliser of $X$ acts by multiplying the `group part' in the second 
factor by a certain word over $F\cup F^{-1}$ representing an element of the subgroup $H$. Hence, 
if $L_H\subseteq (F\cup F^{-1})^\ast$ is a set of representatives for the subgroup $H$, the words 
$\rho^{(2)}(1,w,1)$, $w\in L_H$, constitute a cross-section of the stabiliser with respect to the 
Sch\"{u}tzenberger congruence. It follows by Lemma \ref{lem-w1w2} that the required Sch\"{u}tzenberger 
group is just isomorphic to the group of right translations of $H$ by elements of $H$, which is in turn 
isomorphic to $H$.
\end{proof}

With the view to encouraging future research in this direction, we propose the following two questions.


\begin{que}
Given a biordered set $\E$, explore the relationships between the Sch\"{u}tz\-en\-berger groups associated with non-regular 
$\D$-classes of $\ig{\E}$ and its maximal subgroups.
\end{que}

\begin{que}
Let $\E$ be a finite biordered set. What algorithmic properties of maximal subgroups, Sch\"{u}tzenberger groups, and their 
relationships in $\ig{\E}$ should be assumed to be decidable in order to deduce that the word problem of $\ig{\E}$ is decidable?
\end{que}


\begin{acknowledgements}
The authors are grateful to Stuart W. Margolis (Bar-Ilan), Armando Martino (Southampton), and John C.\ Meakin (Nebraska-Lincoln)
for useful and stimulating discussions and correspondence. Particularly influential were conversations with Stuart Margolis,
concerning short products in $\ig{\E}$ and their significance for both the regularity problem and the word problem. Distinct
echos of these conversations are discernible in both Theorem \ref{thm:reg} and in the ideas underlying the specific construction
in the second half of the paper. 
We are also grateful to an anonymous referee for their careful reading of the article
and helpful suggestions for improving the exposition in many places.
The first named author gratefully acknowledges the kind hospitality of the School of 
Mathematics, University of East Anglia, Norwich, and the School of Mathematics and Statistics, University of St Andrews.
Finally, the second and the third named authors thank the Centre of Algebra of the University of Lisbon (CAUL) for providing
great conditions for working on this subject.
\end{acknowledgements}


\end{document}